\documentclass[aop,MSNbibl,citesort,dvips]{arximspdf}

\doi{10.1214/11-AOP679}
\volume{40}
\issue{5}
\pubyear{2012}
\firstpage{2264}
\lastpage{2297}

\makeatletter

\newtheorem{theorem}{Theorem}[section]
\newtheorem{corollary}{Corollary}[theorem]
\newtheorem{lemma}{Lemma}[section]
\newproclaim{definition}{Definition}[section]

\newproclaim{remark}{Remark}[section]

\makeatother

\begin{document}
\begin{frontmatter}

\title{Existence, uniqueness and comparisons for BSDEs in~general spaces}
\runtitle{BSDEs in general spaces}

\begin{aug}
\author[A]{\fnms{Samuel N.} \snm{Cohen}\corref{}\thanksref{t1}\ead[label=e1]{samuel.cohen@maths.ox.ac.uk}\ead[label=u1,url]{http://people.maths.ox.ac.uk/cohens/}} and
\author[B]{\fnms{Robert J.} \snm{Elliott}\ead[label=e2]{relliott@ucalgary.ca}\ead[label=u2,url]{http://people.ucalgary.ca/\textasciitilde relliott/}}

\runauthor{S. N. Cohen and R. J. Elliott}
\affiliation{University of Adelaide, University of Adelaide and
University of Calgary}
\address[A]{Mathematical Institute\\
University of Oxford\\
OX1 3LB, Oxford\\
United Kingdom\\
\printead{e1}\\
\printead{u1}}
\address[B]{School of Mathematical Sciences\\
University of Adelaide\\
Adelaide, South Australia, 5005\\
\printead{e2}\\
\printead{u2}}
\end{aug}

\thankstext{t1}{Samuel Cohen is now at the University of Oxford;
this work was completed at the University of Adelaide.}

\received{\smonth{1} \syear{2010}}
\revised{\smonth{2} \syear{2011}}

%
\begin{abstract}
We present a theory of backward stochastic differential equations in
continuous time with an arbitrary filtered probability space. No
assumptions are made regarding the left continuity of the filtration,
of the predictable quadratic variations of martingales or of the
measure integrating the driver. We present conditions for existence and
uniqueness of square-integrable solutions, using Lipschitz continuity
of the driver. These conditions unite the requirements for existence in
continuous and discrete time and allow discrete processes to be
embedded with continuous ones. We also present conditions for
a~comparison theorem and hence construct time consistent nonlinear
expectations in these general spaces.
\end{abstract}

%

\pdfkeywords{60H20, 60H10, 91B16, BSDE,
comparison theorem, general filtration,
separable probability space,
Gronwall inequality,
nonlinear expectation}

\vspace*{12pt}
\end{frontmatter}

\section{Introduction}
The theory of backward stochastic differential equations (BSDEs) has
been extensively studied. Typically, results have been obtained only in
the context of a filtration generated by a Brownian motion, possibly
with the addition of Poisson jumps. Specifically, attention has been
given to equations of the form
\[
dY_t = F(\omega, t, Y_{t-}, Z_t)\,dt - Z_t^* \,dM_t, \qquad  Y_T = Q,
\]
where $M$ is the martingale generating the filtration (typically
Brownian motion), $T$ is a fixed finite terminal time, $Q\in
L^2(\mathcal{F}_T)$ is a stochastic terminal value, $F$ is a
progressively measurable function, $[\cdot]^*$ denotes matrix/vector
transposition (and hence $A^*B$ denotes the inner product of $A$ and
$B$) and the solution is a square integrable pair of processes $(Y,Z)$,
where $Y$ is adapted and $Z$ is predictable.

A notable exception to this is the work of El Karoui and Huang~\cite{El1997a}, where a general probability space is considered. In the case
considered in~\cite{El1997a}, the martingale $M$ is specified a priori,
and the equation considered is
%
\begin{equation}\label{eqBSDEworthogonal}
dY_t = F(\omega, t, Y_{t-}, Z_t)\,dC_t - Z_t^* \,dM_t- dN_t; \qquad  Y_T = Q,
\end{equation}
where each term is as above, the filtration is quasi-left continuous,
$C$ is a~continuous process such that $d\langle M\rangle$ is absolutely
continuous with respect to $dC$ and $N$ is a martingale strongly
orthogonal to $M$, that is, $\langle M, N\rangle=0$, where $\langle
\cdot, \cdot\rangle$ denotes the predictable quadratic covariation process.

These equations depend heavily on the continuity of $C$ and, therefore,
are unable to deal with any situation where martingales may jump at a
point with positive probability. However, these situations may arise in
various applications. For example, when using BSDEs in modeling
dividend paying assets, the martingales involved may jump at the time
of the dividend announcement. Similarly, if we consider embedding a
discrete time process in continuous time, we obtain processes which
jump with positive probability at every integer.

A significant use of these equations is to generate ``nonlinear
expectations'' or ``nonlinear evaluations,'' in the sense of~\cite{Peng2004}. These are operators
\[
\mathcal{E}(\cdot|\mathcal{F}_t)\dvtx L^2(\mathcal{F}_T)\to L^2(\mathcal{F}_t),
\]
satisfying certain basic properties. They have important applications
in mathematical finance and stochastic control. Given the results of
\cite{Coquet2002} and~\cite{Hu2008}, it is known that in the Brownian
setting, under certain conditions, these operators are completely
described by BSDEs. Furthermore, it is clear, given the comparison
theorem in~\cite{Cohen2009}, BSDEs of the form of (\ref
{eqBSDEworthogonal}) in arbitrary spaces, under some conditions, also
describe nonlinear expectations. However, it is not known how large a
class of nonlinear expectations in a general space is given by a BSDE.

To establish such a result for BSDEs of the form of (\ref
{eqBSDEworthogonal}), one faces a~significant problem. If $\mathcal
{E}(Q|\mathcal{F}_t)=Y_t$ is given as the solution to (\ref
{eqBSDEworthogonal}) for some~$F$ not dependent on $Y_{t-}$, once $M$
is fixed, for any martingale $N$ orthogonal to~$M$ with $N_0=0$, we have the property
\[
\mathcal{E}(Q+N_T|\mathcal{F}_t) = \mathcal{E}(Q|\mathcal{F}_t).
\]
This property is clearly not true for most nonlinear expectations,
whenever there are nontrivial examples of such processes $N$, which is
not the case in the Brownian setting (as a martingale representation
theorem holds). It follows that these equations cannot describe any
nonlinear expectations which do not possess this property.

Furthermore, the fact that the martingale $M$ must be specified a
priori is arguably unsatisfying. Conceptually, it may be preferable if,
in some sense, the probability space itself dictated what martingales
are needed for the BSDE.\vadjust{\goodbreak} In this case, one could proceed either by
specifying the probability space using a collection of martingales
(which, given a representation theorem holds, will then describe all
martingales in the space), or vice versa.

In this paper we establish such a general result. We show that there is
a sense in which the original BSDE can be interpreted in a general
space, using only a separability assumption on $L^2(\mathcal{F}_T)$. We
establish conditions on the existence and uniqueness of BSDEs in this
setting, where the driver is integrated with respect to an arbitrary
deterministic Stieltjes measure (Theorem~\ref{thmBSDEExist2}). We also
prove a comparison theorem for these solutions, which shows under which
conditions they do indeed describe nonlinear expectations and evaluations.

A similar approach is taken in~\cite{Hassani2002}, where a form of BSDE
is considered using generic maps from a space of semimartingales to the
spaces of square-integrable martingales and of finite-variation
processes integrable with respect to a given continuous increasing
process. Using Browder's theorem, they demonstrate the existence of
solutions to these equations on an infinite horizon. Our approach
differs from theirs by considering a classical form of BSDE on a finite
horizon and deriving an existence result using a contraction mapping
technique. Because of this, our conditions for existence are a more
straightforward extension of those in the classical case. More
significantly, our approach does not require the driver of the BSDE to
be integrated with respect to a continuous measure, which allows a
unification of the discrete and continuous time theory of BSDEs.

\section{Martingale representations}
The key result used in the construction of BSDEs is the Martingale
representation theorem. In the Brownian setting, this result is well
known (see, e.g.,~\cite{Revuz1999}, Chapter V.3, or
\cite{Elliott1982}, Theorem 12.33). In other cases, for example, when dealing with
martingales generated by Markov chains, a similar result is available
(see~\cite{Cohen2008}); however it is also known that there exist
probability spaces in which no finite-dimensional martingale
representation theorem exists.

Consider a probability space $(\Omega, \mathcal{F}, \mathbb{P})$ with a
filtration $\{\mathcal{F}_t\}, t\in[0,T]$, satisfying the usual
conditions of completeness and right continuity. The time-interval
$[0,T]$ is given the Borel $\sigma$-field $\mathcal{B}([0,T])$.
\begin{definition}\label{defninducedmeasure}
For any nondecreasing process of finite variation $\mu$, we define the
measure induced by $\mu$ to be the measure over $\Omega\times[0,T]$
given by
\[
A\mapsto E\biggl[\int_{[0,T]}I_A(\omega, t) \,d\mu\biggr].
\]
Here $A\in\mathcal{F}\otimes\mathcal{B}([0,T])$, and the integral is
taken pathwise in a Stieltjes sense.
\end{definition}
\begin{remark} \label{reminducedmeasurerem}
If $\mu$ is a deterministic process, then this definition gives the
product measure $\mu\times\mathbb{P}$. We can also consider these as
measures on the space $(\Omega\times[0,T], \mathcal{P})$, where
$\mathcal{P}$ is the predictable $\sigma$-algebra.\vadjust{\goodbreak}
\end{remark}

Under the assumption that the Hilbert space $L^2(\mathcal{F}_T)$ is
separable, a paper of Davis and Varaiya~\cite{Davis1974} gives the
following result (see also Malamud~\cite{Malamud2007}).\vspace*{-1pt}

\begin{theorem}[(Martingale representation theorem;~\cite{Davis1974})]
\label{thmmartrep}
Suppose that $L^2(\mathcal{F}_T)$ is a separable Hilbert space, with an
inner product $(X,Y) = E[XY]$. Then there exists a finite or countable
sequence of square-integrable $\{\mathcal{F}_t\}$-martingales $M^1,
M^2,\ldots$ such that every square integrable $\{\mathcal{F}_t\}
$-martingale~$N$ has a representation
\[
N_t= N_0+\sum_{i=1}^{\infty} \int_{]0,t]}Z^i_u \,dM^i_u
\]
for some sequence of predictable processes $Z^i$. This sequence satisfies
%
\begin{equation} \label{eqZbound}
E\Biggl[\sum_{i=0}^\infty\int_{[0,T]} (Z^i_u)^2 \,d\langle M^i\rangle
_u\Biggr] <+\infty.
\end{equation}

These martingales are orthogonal (i.e., $E[M^i_TM^j_T]=0$ for all
$i\neq j$), and the predictable quadratic variation processes $\langle
M^i\rangle$ satisfy
\[
\langle M^1\rangle\succ\langle M^2\rangle\succ\cdots,
\]
where $\succ$ denotes absolute continuity of the induced measures
(Definition~\ref{defninducedmeasure}). Furthermore, these martingales
are unique, in that if $N^i$ is another such sequence, then $\langle
N^i\rangle\sim\langle M^i\rangle$, where $\sim$ denotes equivalence
of the induced measures.\vspace*{-1pt}
\end{theorem}
\begin{corollary}
For any predictable processes $Z^i$ satisfying~(\ref{eqZbound}), the
process $\sum_i\int_{]0,t]} Z^i_u\,dM_u^i$ is well defined and is a
square-integrable martingale.\vspace*{-1pt}
\end{corollary}

\begin{remark}
When a finite-dimensional martingale representation theorem holds, as
when the space is generated\vspace*{1pt} by a Brownian motion, then all but finitely
many of the martingales $M^i$ given by Theorem~\ref{thmmartrep} will
be zero. We shall not, in general, assume that this is the case, but
acknowledge that, in this situation, significant simplification of the
equations considered is possible.\vspace*{-1pt}
\end{remark}

We shall use this result to construct a form of BSDE on this general space.\vspace*{-1pt}

\begin{definition}
We denote\vspace*{1pt} by $\mathbb{R}^{K\times\infty}$ the space of infinite
$\mathbb{R}^K$-valued sequences. We note that the predictable processes
$Z^i$ in Theorem~\ref{thmmartrep} can be written as a vector process
$Z$, which takes values in $\mathbb{R}^{1\times\infty}$.\vspace*{-1pt}
\end{definition}

\section{BSDEs in general spaces: A definition}
We seek to construct BSDEs, assuming only the usual properties of the
filtration and that $L^2(\mathcal{F}_T)$ is a~separable Hilbert space.
For simplicity, we shall also assume that $\mathcal{F}_0$ is trivial,
which, by right continuity, ensures that, almost surely, no martingale
has a~jump at $t=0$.\vadjust{\goodbreak}
\begin{definition}
Let $\mu$ be a deterministic signed Stieltjes
measure. For \mbox{$K\in\mathbb{N}$}, a BSDE is an equation of the form
%
\begin{equation}\label{eqBSDE}
Q = Y_t - \int_{]t,T]} F(\omega, u, Y_{u-}, Z_u) \,d\mu_u + \sum
_{i=1}^\infty\int_{]t,T]}Z^i_u\,dM^i_u,
\end{equation}
where $Z_t(\omega)$ is the (countably infinite) vector with entries
$\{Z^i_t(\omega)\in\mathbb{R}^K\}_{i\in\mathbb{N}}$. For a~terminal value $Q\in
L^2(\mathbb{R}^K;\mathcal{F}_T)$, a predictable \textit{driver} function
$F\dvtx\Omega\times[0,T]\times\mathbb{R}^K\times
\mathbb{R}^{K\times\infty}\to\mathbb{R}^K$, a solution is a pair of processes
$(Y, Z)$ taking values in $\mathbb{R}^K\times\mathbb{R}^{K\times\infty
}$, where
$Z$ is predictable, and $Y$ is adapted. We shall restrict our attention
to the
case when $Y$ is square integrable, and $Z$
satisfies~(\ref{eqZbound}).\vspace*{-2pt}
\end{definition}
\begin{remark}
We note that this type of equation encompasses most previously studied
forms of
BSDEs. When the filtration is Brownian, we can take $M^i$ to be the $i$th
component of the generating Brownian motion, $\mu=t$, and the equation is
standard. When the filtration is generated by a Poisson random measure
over a
separable space and a Brownian motion, as in~\cite{Barles1997,Royer2006}
and others, or by a Markov chain, as in~\cite{Cohen2008,Cohen2008b}, we
have a similar reduction. When we consider the analogous equations in discrete
time, we can form the discrete-time filtration embedded in this
continuous time
context (see~\cite{Jacod2003}, Chapter 1f) and hence obtain the
backward stochastic
difference equations considered in~\cite{Cohen2008c} and~\cite{Cohen2009a}.

Comparing with the work of~\cite{El1997a}, we see that if $F$ depends
only on the projection of $Z$ into a finite-dimensional subspace of
$\mathbb{R}^{K\times\infty}$, then it is possible to reduce the
equation to a form similar to~(\ref{eqBSDEworthogonal}).\vspace*{-2pt}
\end{remark}

We shall present a result (Theorem~\ref{thmBSDEExist2}) demonstrating
conditions under which there exists a unique solution to such an equation.\vspace*{-2pt}

\section{Inequalities for Stieltjes integrals}
To give conditions under which solutions to a BSDE exist, we must first
establish the following results regarding integrals with respect to Stieltjes
measures. These results are standard whenever the measures are continuous.\vspace*{-2pt}

\subsection{Stieltjes exponentials}\vspace*{-2pt}
\begin{definition}\label{defexponentials}
For any c\`adl\`ag function of finite variation \mbox{$\nu\dvtx[0,\infty[\ \to
\mathbb{R}$}, we write
\[
\mathfrak{E}(\nu_t):=e^{\nu_t}\prod_{0\leq s\leq t}(1+\Delta\nu
_s)e^{-\Delta\nu_s},
\]
and call this the \textit{Stieltjes exponential} of $\nu$. Note that this
is also a c\`adl\`ag function.\vspace*{-2pt}
\end{definition}

Note that $\mathfrak{E}(\nu_t)$ should be more properly written as
$\mathfrak{E}(\nu_{(\cdot)};t)$, as it is a~function of $\{\nu_s;s\leq
t\}$ not just of $\nu_t$. We use the\vadjust{\goodbreak} former notation purely for
compactness, whenever this does not lead to confusion. We note the
following useful bound.\vspace*{-2pt}
\begin{lemma} \label{lemprocessexpbound}
If $\nu$ is a c\`adl\`ag function, then $\mathfrak{E}(\nu_t) \leq e^{\nu
_t}$, where $e^{\nu_t}$ is the classical exponential of $\nu_t$.\vspace*{-2pt}
\end{lemma}
\begin{pf}
As $e^x\geq1+x$, it is clear that $(1+\Delta\nu_t)e^{-\Delta\nu
_t}\leq1$ for all $t$. The result follows.\vspace*{-2pt}
\end{pf}
\begin{lemma}\label{lemexpproperties}
For any c\`adl\`ag function of finite variation, the Stieltjes~exponential is
well defined. Furthermore, if $\Delta\nu_t\!\geq\!-1$, then
$\mathfrak{E}(\nu_t)\!\geq\!0$. If $\Delta\nu_s \!>\!-1$, then $\mathfrak
{E}(\nu_t)\!>\!0$, and $\mathfrak{E}(\nu_t)^{-1}$ is well defined. In this
case, the process $u_t\!=\break u_s \mathfrak{E}(\nu_t)\mathfrak{E}(\nu
_s)^{-1}$ is the solution to the Lebesgue--Stieltjes integral equation,
\[
u_t = u_{s}+\int_{]s,t]} u_{r-} \,d\nu_r.\vspace*{-2pt}
\]
\end{lemma}
\begin{pf}
As the process $\nu_t$ is c\`adl\`ag and of finite variation, it is a
(deterministic) semimartingale. $\mathfrak{E}(\nu_t)$ is then the
standard Dol\'eans--Dade exponential of this process, and so its
existence and basic properties can be seen in~\cite{Elliott1982},
Theorem~13.5~ff. This guarantees the convergence of the infinite
products considered and solves the desired integral equation. The
nonnegativity result is clear by inspection.

For the positivity result, we need only show that $\prod_{0\leq s\leq
t}(1+\Delta\nu_s)>0$. By continuity of the logarithm, this is
equivalent to showing that
\[
-\sum_{0\leq s\leq t} \log(1+\Delta\nu_s)<\infty.
\]

We then note that we can consider three cases. First, if $\Delta\nu
_s\geq0$, then $-\log(1+\Delta\nu_s)\leq0$, and hence
\[
-\biggl(\sum_{\{0\leq s\leq t\}\cap\{\Delta\nu_s\geq0\}} \log(1+\Delta\nu_s)\biggr)
\leq0 <\infty.
\]
Second, we note that $\sum_{0<s\leq t}|\Delta\nu_s|$ is finite, as $\nu
$ is of finite variation, and hence there are only finitely many $s$
such that $\Delta\nu_s\leq-0.7$. Therefore
\[
-\biggl(\sum_{\{0\leq s\leq t\}\cap\{\Delta\nu_s \leq-0.7\}}\log(1+\Delta\nu
_s)\biggr) <\infty.
\]
Finally, we know that $2x<\log(1+x)<0$ for $-0.7<x<0$. Hence, we have
\begin{eqnarray*}
-\biggl(\sum_{\{0\leq s\leq t\}\cap\{-0.7<\Delta\nu_s <0\}}\log(1+\Delta
\nu_s)\biggr) &<& \biggl(\sum_{\{0\leq s\leq t\}\cap\{-0.7<\Delta\nu_s
<0\}}2|\Delta\nu_s|\biggr) \\
&<&\infty.
\end{eqnarray*}
Combining these three sums gives the desired constraint on the
logarithm, and hence the strict positivity of the desired
product.\vadjust{\goodbreak}
\end{pf}
\begin{lemma}\label{lemexpinfpositive}
For $\nu$ a c\`adl\`ag function of finite variation with $\Delta\nu
_t>-1$, we have the stronger result
\[
\inf_{0\leq t\leq T}\biggl\{ \prod_{0\leq s\leq t}(1+\Delta\nu_s)
\biggr\} >0.
\]
\end{lemma}
\begin{pf}
By the same argument as in Lemma~\ref{lemexpproperties}, we have
\[
-\biggl(\sum_{\{0\leq s\leq T\}\cap\{\Delta\nu_s <0\}} \log(1+\Delta\nu
_s)\biggr)<\infty.
\]
It follows that
\[
-\sum_{0\leq s\leq t} \log(1+\Delta\nu_s)<-\biggl(\sum_{\{0\leq s\leq T\}
\cap\{\Delta\nu_s <0\}} \log(1+\Delta\nu_s)\biggr)<\infty
\]
for all $t$. Hence
\[
\inf_{0\leq t\leq T} \biggl\{\prod_{0\leq s\leq t}(1+\Delta\nu_s)
\biggr\} >\biggl(\prod_{\{0\leq s\leq T\}\cap\{\Delta\nu_s<0\}}(1+\Delta\nu
_s)\biggr)>0.
\]
\upqed\end{pf}
\begin{definition}\label{defnjumpinversion}
Let $\nu$ be a c\`adl\`ag function of finite variation with $\Delta\nu
_t>-1$ for all $t$. Then the \textit{left-jump inversion} of $\nu$ is
defined by
\[
\bar\nu_t = \nu_t - \sum_{0\leq s\leq t} \frac{(\Delta\nu
_s)^2}{1+\Delta\nu_s}.
\]
Similarly if $\Delta\nu_t<1$ for all $t$, the \textit{right-jump
inversion} is defined by
\[
\tilde\nu_t = \nu_t + \sum_{0\leq s\leq t} \frac{(\Delta\nu
_s)^2}{1-\Delta\nu_s}.
\]
\end{definition}
\begin{lemma}
For $\nu$ a function as in Definition~\ref{defnjumpinversion}, the
left- and right-jump inversions are finite (whenever they are
defined), and satisfy
\[
\mathfrak{E}(\nu_t)^{-1} = \mathfrak{E}(-\bar\nu_t)
\]
and
\[
\mathfrak{E}(-\nu_t) = \mathfrak{E}(\tilde\nu_t)^{-1}.
\]
\end{lemma}
\begin{pf}
Consider first the left-jump-inversion. We know that $\Delta\nu_s>-1$ and
$\sum|\Delta\nu_s|<\infty$. Hence it follows that $\Delta\nu_s$ has only
finitely many values in any neighborhood not containing zero and hence is
bounded away from $-1$. That is, there exists some $\varepsilon>0$ such that
$\Delta\nu_s> \varepsilon-1$ for all $s$. To show finiteness, write
\[
\sum_{\{0\leq s\leq t\}\cap\{\Delta\nu_s \geq0\}} \frac{(\Delta\nu_s)^2}
{1+\Delta\nu_s} \leq\sum_{\{0\leq s\leq t\}\cap\{\Delta\nu_s \geq0\}}
|\Delta
\nu_s|<\infty\vadjust{\goodbreak}
\]
and
\begin{eqnarray*}
\sum_{\{0\leq s\leq t\}\cap\{\Delta\nu_s < 0\}} \frac
{(\Delta\nu_s)^2}{1+\Delta\nu_s} &\leq& \varepsilon^{-1}\biggl(\sum_{\{
0\leq s\leq t\}\cap\{\Delta\nu_s < 0\}} (\Delta\nu_s)^2\biggr)\\
&<& \varepsilon^{-1}\biggl(\sum_{\{0\leq s\leq t\}\cap\{\Delta\nu_s < 0\}}
|\Delta\nu_s|\biggr)\\
&<&\infty.
\end{eqnarray*}
Combining these sums gives the desired finiteness result.

We now note that, algebraically,
\[
(1-\Delta\bar\nu_s)^{-1}=\biggl(1-\Delta\nu_s + \frac{(\Delta\nu
_s)^2}{1+\Delta\nu_s}\biggr)^{-1} = 1+\Delta\nu_s.
\]
Hence
\begin{eqnarray*}
\mathfrak{E}(\nu_t)^{-1}&=&e^{-\nu_t}\prod_{0\leq s\leq t}(1+\Delta\nu
_s)^{-1}e^{\Delta\nu_s}\\
&=&e^{-\nu_t+\sum_{0<s\leq t}({(\Delta\nu_s)^2}/({1+\Delta\nu
_s}))}\prod_{0\leq s\leq t}(1+\Delta\nu_s)^{-1}e^{\Delta\nu
_s-{(\Delta\nu_s)^2}/({1+\Delta\nu_s})}\\
&=&e^{-\bar\nu_t}\prod_{0\leq s\leq t}(1-\Delta\bar\nu_s)e^{\Delta\bar
\nu_s}\\
&=&\mathfrak{E}(-\bar\nu_t).
\end{eqnarray*}

The proof for the right-jump inversion follows in the same way, where
finiteness is because
\[
\sum_{0\leq s\leq t} \frac{(\Delta\nu_s)^2}{1-\Delta\nu_s} = \sum
_{0\leq s\leq t} \frac{(-\Delta\nu_s)^2}{1+(-\Delta\nu_s)},
\]
and $-\nu_s$ satisfies the requirements given above for the left-jump
inversion. The algebraic result is then that
\[
(1+\Delta\tilde\nu_s)^{-1} = \biggl(1+\Delta\nu_s + \frac{(\Delta\nu
_s)^2}{1-\Delta\nu_s}\biggr)^{-1} = 1-\Delta\nu_s,
\]
and the result is as given.
\end{pf}
\begin{lemma}
For $\nu$ a c\`adl\`ag function of bounded variation with $\Delta\nu
_s>-1$, the right-jump inversion of the left-jump inversion of $\nu$ is
the original function, that is,
\[
\tilde{\bar{\nu}}_t = \nu_t.
\]
Similarly, if $\Delta\nu_s<-1$, then $\bar{\tilde{\nu}}_t=\nu_t$.
\end{lemma}
\begin{pf}
For simplicity, we decompose $\nu$ into a discontinuous part $\nu^d_t:=
\sum_{0\leq s\leq t} \Delta\nu_s$ and a continuous part $\nu^c_t=\nu
_t-\nu_d$. Clearly,\vspace*{1pt} taking either the left- or right-jump inversion
will not alter the continuous part $\nu^c$, and so it is sufficient to
show that the discontinuous parts are equal, that is, $\Delta\tilde{\bar
\nu}_t= \Delta\bar{\tilde\nu}_t = \Delta\nu_t$ for all $t$, whenever
these terms are well defined. From Definition~\ref{defnjumpinversion}
we have
\[
\Delta\bar\nu_t = \frac{\Delta\nu_t}{1+\Delta\nu_t},\qquad \Delta\tilde
\nu_t = \frac{\Delta\nu_t}{1-\Delta\nu_t}
\]
and hence
\[
\Delta\tilde{\bar\nu}_t = \frac{\Delta\bar\nu_t}{1-\Delta\bar\nu_t} =
\frac{{\Delta\nu_t}/({1+\Delta\nu_t})}{1-{\Delta\nu_t}/({1+\Delta
\nu_t})} = \Delta\nu_t,
\]
and similarly $\Delta\bar{\tilde\nu}_t = \Delta\nu_t$, as desired.\vspace*{-2pt}
\end{pf}

\subsection{Integrating factors}
It is useful to have some results relating to the solutions of
equations of the form $du_t - u_{t-}\,d\nu_t = \cdots.$ These are similar
the the classical results on the use of integrating factors and Gr\"
onwall's inequality in the study of ordinary differential
equations.\vspace*{-2pt}
\begin{definition}
Let $u, v$ be two measures on a $\sigma$-algebra $\mathcal{A}$. We
write $du\leq dv$ if, for any $A\in\mathcal{A}$, $u(A) \leq v(A)$.\vspace*{-2pt}
\end{definition}
\begin{remark}
When $v$ is a nonnegative measure, and $u$ is absolutely continuous
with respect to $v$, this definition is equivalent to requiring that
the Radon--Nikodym derivative satisfies $du/dv\leq1$, $dv$-a.e.\vspace*{-2pt}
\end{remark}
\begin{lemma} \label{lemintegratingfactorbound}
Let $u$, $\nu$ and $w$ be signed Stieltjes measures on $\mathcal
{B}([0,T])$, such that $\Delta\nu_t < 1$ for all $t$, and
\[
du_t \geq- u_{t-} \,d\nu_t + dw_t,
\]
then
\[
d(u_t \mathfrak{E}(\tilde\nu_t)) \geq(1-\Delta\nu_t)^{-1}\mathfrak
{E}(\tilde\nu_{t-})\,dw_t,
\]
where $\tilde\nu$ is the right-jump inversion of $\nu$.\vspace*{-2pt}
\end{lemma}
\begin{pf}
Applying the product rule for Stieltjes integrals we have
\[
\frac{d(u_t \mathfrak{E}(\tilde\nu_t))}{\mathfrak{E}(\tilde\nu_{t-})}
= du_t+u_{t-}\,d\tilde\nu_t + \Delta u_t\Delta\tilde\nu_t.
\]
As $d\tilde\nu_t = d\nu/(1-\Delta\nu_t)$ and $\Delta u_t\Delta\nu_t=
(\Delta\nu_t) \,du_t$, this gives
\begin{eqnarray*}
\frac{d(u_t \mathfrak{E}(\tilde\nu_t))}{\mathfrak{E}(\tilde\nu_{t-})}
&=& du_t+u_{t-}\frac{d\nu_t}{1-\Delta\nu_t} + \frac{\Delta\nu_t}{1-\Delta
\nu_t}\,du_t\\
&=& \biggl(1+\frac{\Delta\nu_t}{1-\Delta\nu_t}\biggr)\,du_t+u_{t-}\frac{d\nu
_t}{1-\Delta\nu_t}\\
&=& (1-\Delta\nu_t)^{-1} (du_t + u_{t-}\,d\nu_t)\\
&\geq&(1-\Delta\nu_t)^{-1} \,dw_t.
\end{eqnarray*}
\upqed\end{pf}
\begin{lemma}[(Backward Gr\"onwall inequality)]
\label{lemgroenwallinequality}
Let $u$ be a process such that, for $\nu$ a nonnegative Stieltjes
measure with $\Delta\nu_t<1$ and $\alpha$ a $\tilde\nu$-integrable
process, $u$ is $\nu$-integrable and
\[
u_{t} \leq\alpha_t + \int_{]t,T]} u_{s} \,d\nu_s,
\]
then
\[
u_{t} \leq\alpha_t + \mathfrak{E}(-\nu_t)\int_{]t,T]}\mathfrak
{E}(\tilde\nu_{s}) \alpha_s \,d\tilde\nu_s.
\]
If $\alpha_t=\alpha$ is constant, this simplifies to
\[
u_{t} \leq\alpha\mathfrak{E}(\tilde\nu_T) \mathfrak{E}(\tilde\nu
_t)^{-1}=\alpha\mathfrak{E}(-\nu_t)\mathfrak{E}(-\nu_T)^{-1}.
\]
\end{lemma}
\begin{pf}
First note that $d\nu_t = \frac{d\tilde\nu}{1+\Delta\tilde\nu_t}$ and
that $\Delta\tilde\nu_t\Delta\nu_t=\Delta\tilde\nu_t\,d\nu_t$. Then~let
\[
w_t := \mathfrak{E}(\tilde\nu_t)\int_{]t,T]} u_{s} \,d\nu_s.
\]
From the product rule for stochastic integrals, as $\nu$ is of finite variation,
\begin{eqnarray*}
\frac{dw_t}{\mathfrak{E}(\tilde\nu_{t-})} &=& \biggl(\int_{]t,T]} u_{s}
\,d\nu_s\biggr)\,d\tilde\nu_s - u_{t} \,d\nu_t - u_{t}\Delta\nu_t\Delta\tilde
\nu_t\\
&=& -u_{t} (1+\Delta\tilde\nu_t)\,d\nu_t + \biggl(\int_{]t,T]} u_{s} \,d\nu
_s\biggr)\,d\tilde\nu_t\\
&=& -u_{t} \,d\tilde\nu_t + \biggl(\int_{]t,T]} u_{s} \,d\nu_s\biggr)\,d\tilde
\nu_t\\
&=& \biggl(-u_{t} + \int_{]t,T]} u_{s} \,d\nu_s\biggr)\,d\tilde\nu_t\\
&\geq&-\alpha_t\,d\tilde\nu_t.
\end{eqnarray*}
Note that $d\tilde\nu_t$ and $\mathfrak{E}(\tilde\nu_{t-})$ are both
nonnegative. Therefore, by integration,
\[
w_t = \mathfrak{E}(\tilde\nu_t)\int_{]t,T]} u_{s} \,d\nu_s \leq\int
_{]t,T]} \mathfrak{E}(\tilde\nu_{s-}) \alpha_s\,d\tilde\nu_s.
\]
Substitution yields
\[
u_{t} \leq\alpha_t + \mathfrak{E}(\tilde\nu_t)^{-1}\int
_{]t,T]}\mathfrak{E}(\tilde\nu_{s-}) \alpha_s \,d\tilde\nu_s,\vadjust{\goodbreak}
\]
and the desired inequalities follow from $\mathfrak{E}(\tilde\nu
_t)^{-1}=\mathfrak{E}(-\nu_t)$.
If $\alpha_t=\alpha$, then this simplifies to
\begin{eqnarray*}
u_{t-} &\leq&\alpha\biggl[ 1 + \mathfrak{E}(\tilde\nu_t)^{-1}\int
_{]t,T]}\mathfrak{E}(\tilde\nu_{s-}) \,d\tilde\nu_s\biggr]\\
&=& \alpha\bigl[ 1 + \mathfrak{E}(\tilde\nu_t)^{-1}\bigl(\mathfrak{E}(\tilde
\nu_{T})-\mathfrak{E}(\tilde\nu_{t})\bigr)\bigr]\\
&=& \alpha\mathfrak{E}(\tilde\nu_{T})\mathfrak{E}(\tilde\nu_t)^{-1}.
\end{eqnarray*}
\upqed\end{pf}
\begin{lemma}[(Forward Gr\"onwall inequality)]
Let $u$ be a function such that, for $\nu$ a nonnegative Stieltjes
measure and $\alpha$ a $\bar\nu$-integrable process, $u$~is $\nu
$-integrable and
\[
u_{t} \leq\alpha_t + \int_{]0,t]} u_{s} \,d\nu_s,
\]
then
\[
u_{t} \leq\alpha_t + \mathfrak{E}(\nu_t)\int_{]0,t]}\mathfrak{E}(-\bar
\nu_{s}) \alpha_s \,d\bar\nu_s.
\]
If $\alpha_t=\alpha$ is constant, this simplifies to
\[
u_{t} \leq\alpha\mathfrak{E}(\nu_t).
\]
\end{lemma}
\begin{pf}
This result follows in an almost identical fashion to Lemma \ref
{lemgroenwallinequality}, and the proof is therefore omitted.
\end{pf}

\section{Existence of BSDE solutions: Fundamental results}

In this section we shall establish the existence of solutions to BSDEs
when the process $\mu$ satisfies particular properties.
\begin{definition}\label{defnmudefn}
Let $\mu$ be a deterministic nondecreasing right-continuous function
$\mu\dvtx[0,T]\to\mathbb{R}^+$. The measure $d\mu$ will serve in the place
of the Lebesgue measure $dt$ in our BSDE.

As $\mu$ is of finite variation, its discontinuities $\Delta\mu$ are
bounded. We assume that $\mu$ assigns positive measure to any nonempty
open interval in $[0,T]$.
\end{definition}

Unless otherwise indicated, all (in-)equalities should be read as ``up
to evanescence.''
\begin{definition}
We denote by \mbox{$\|\cdot\|$} the standard Euclidean norm on~$\mathbb{R}^K$,
and note that $\|y\|^2= y^*y$, where $[\cdot]^*$ denotes vector transposition.
\end{definition}
\begin{definition}\label{defnMnorm}
For a given $\mu$ and fixed $K\in\mathbb{N}$, we define the stochastic
seminorm \mbox{$\|\cdot\|_{M_t}$} on $\mathbb{R}^{K\times\infty}$ as follows.
For each $i\in\mathbb{N}$, consider $\langle M^i\rangle$ as a~measure
on the predictable $\sigma$-algebra; cf. Remark \ref
{reminducedmeasurerem}. Let $\langle M^i\rangle$ have the Lebesgue
decomposition
\[
\langle M^i\rangle_t = m^{i,1}_t + m^{i,2}_t,\vadjust{\goodbreak}
\]
where $m^{i,1}_t$ is absolutely continuous with respect to $\mu\times
\mathbb{P}$, and $m^{i,2}_t$ is orthogonal to $\mu\times\mathbb{P}$.
As they represent bounded measures on the predictable $\sigma$-algebra,
both $m^{i,1}_t$ and $m^{i,2}_t$ will be nondecreasing predictable processes.

We define, for $z_t\in\mathbb{R}^{K\times\infty}$,
\[
\|z_t\|^2_{M_t} := \sum_i \biggl[\|z_t^i\|^2 \frac{dm^{i,1}_t}{d(\mu
\times\mathbb{P})}\biggr],
\]
where $z^i_t\in\mathbb{R}^K$ is the $i$th element in $z_t$, considered
as a series of values in~$\mathbb{R}^K$.

We note that, for any predictable, progressively measurable process $Z$
taking values in $\mathbb{R}^{K\times\infty}$, and, in particular, for
processes satisfying~(\ref{eqZbound}) in each of their $K$ components,
we have the inequality
%
\begin{eqnarray}\label{eqisometry}
E\biggl[\int_{A}\|Z_t\|^2_{M_t} \,d\mu\biggr] &\leq& E\biggl[\sum_i\int_{A}\|
Z_t^i\|^2\,d\langle M^i_t\rangle\biggr]\nonumber\\
&=&E\biggl[\sum_i\biggl\|\int_{A}Z_t^i\,dM^i_t\biggr\|^2\biggr]\\
&=&E\biggl[
\biggl\|\sum_i\int_{A}Z_t^i\,dM^i_t\biggr\|^2\biggr]
\nonumber
\end{eqnarray}
for any predictable set $A\subseteq\Omega\times[0,T]$. (Note the
latter equalities are simply the standard isometry\vspace*{1pt} used in the
construction of the stochastic integral, by the orthogonality of the $M^i$.)
\end{definition}
\begin{definition} \label{defnprocessspaces}
We define the following spaces of equivalence classes:
\begin{eqnarray*}
H^2_M &=&\biggl\{Z\dvtx\Omega\times[0,T]\to\mathbb{R}^{K\times\infty}\mbox{,
predictable, }\\
&&\hspace*{17.4pt} E\biggl[\sum_i\int_{[0,T]}\|Z_t^i\|^2 \,d\langle
M^i\rangle_t\biggr] <+\infty\biggr\},\\
S^2 &=& \Bigl\{Y\dvtx\Omega\times[0,T]\to\mathbb{R}^K\mbox{,  adapted, }
E\Bigl[\sup_{t\in[0,T]} \|Y_t\|^2\Bigr]<+\infty\Bigr\},\\
H^2_\mu&=& \biggl\{Y\dvtx\Omega\times[0,T] \to\mathbb{R}^K\mbox{,
progressive, }\int_{]0,T]} E[\|Y_{t}\|^2] \,d\mu_t<+\infty\biggr\},
\end{eqnarray*}
where two elements $Z, \bar Z$ of $H^2_M$ are deemed equivalent if
\[
E\biggl[\sum_i\int_{[0,T]}\|Z_t^i-\bar Z_t^i\|^2\,d\langle M^i\rangle
_t\biggr] =0,
\]
two elements of $S^2$ are deemed equivalent if they are
indistinguishable and two elements of $H^2_\mu$ are equivalent if they
are equal $\mu\times\mathbb{P}$-a.s. Note that $K$ is here taken as fixed.
\end{definition}
\begin{remark}
We note that $H^2_M$ is itself a complete metric space, with norm given
by $Z\mapsto E[\sum_i\int_{[0,T]}\|Z_t^i\|^2 \,d\langle M^i\rangle
_t]$; similarly for $H^2_\mu$. Note also that the martingale
representations constructed in Theorem~\ref{thmmartrep} are unique in $H^2_M$.
\end{remark}

A key assumption in the study of BSDEs is the continuity of the driver
function~$F$. When the measure $\mu$ is continuous, we shall show that
it is sufficient that $F$ is uniformly Lipschitz continuous for the
BSDE~(\ref{eqBSDE}) to have a solution. On the other hand, as is clear
in discrete time (cf.~\cite{Cohen2009a}), when $\mu$ is not continuous,
a stronger condition is needed on $F$. We shall call this a~\textit{firm}
Lipschitz bound on $F$, as is defined in the following theorem.
\begin{theorem}\label{thmBSDEExist}
For $\mu$ as in Definition~\ref{defnmudefn}, assume $\mu_T\leq1$. Let
$F\dvtx\Omega\times[0,T]\times\mathbb{R}^K\times\mathbb{R}^{K\times\infty
}\to\mathbb{R}^K$ be a predictable, progressively measurable function
such that:
\begin{itemize}
\item$E[\int_{]0,T]} \|F(\omega, t,
0,0)\|^2\,d\mu_t]<+\infty$;\vspace*{1pt}
\item there exists a linear firm Lipschitz bound on $F$, that is, a
measurable deterministic function $c_t$ uniformly bounded by some $c\in
\mathbb{R}$, such that, for any $y_t, y_t'\in\mathbb{R}^K$, $z_t,
z'_t\in\mathbb{R}^{K\times\infty}$,
\begin{eqnarray*}
&&\|F(\omega, t, y_t, z_t) - F(\omega, t, y'_t, z'_t)\|^2\\
&&\qquad\leq c_t \|
y_t-y'_t\|^2 + c\|z_t-z'_t\|^2_{M_t},\qquad   \,d\mu\times
d\mathbb{P}\mbox{-a.s.}
\end{eqnarray*}
and
\[
c_t\Delta\mu_t<1.
\]
Note that the variable bound $c_t$ need only apply to the behavior of
$F$ with respect to $y$.
\end{itemize}
A function\vspace*{1pt} satisfying these conditions will be called \textit{standard}.
Then for any $Q\in L^2(\mathbb{R}^K;\mathcal{F}_T)$, the BSDE (\ref
{eqBSDE}) with driver $F$ has a unique solution $(Y,Z)\in S^2\times
H^2_M$. ($S^2$ and $H^2_M$ are defined in Definition~\ref{defnprocessspaces}.)
\end{theorem}

To prove this theorem, we first establish the following results.
\begin{lemma} \label{lemH2S2equiv}
If $\mu$ assigns positive measure to every nonempty open interval, then
two c\`adl\`ag processes in $H^2_\mu$ are indistinguishable if and only
if they are equivalent in~$H^2_\mu$. Similarly, two c\`adl\`ag
processes are equivalent in~$H^2_\mu$ if and only if their left limits
are equivalent in~$H^2_\mu$.
\end{lemma}
\begin{pf}
Clearly\vspace*{1pt} indistinguishability implies equivalence of the processes, and
their left limits, in $H^2_\mu$. By right continuity (resp., left
continuity), if on some nonnull set $A$, two processes (resp., their
left limits) differ at any point, they must differ on some nonempty
open interval. As $\mu$\vadjust{\goodbreak} assigns positive measure to such an interval,
it follows that the processes will not be equivalent in $H^2_\mu$.
\end{pf}
\begin{lemma} \label{lemYquadbound}
Let $(Y,Z)$ be the solution to a BSDE with data $(F,Q)$. If~$F$ is
standard, $Q\in L^2(\mathbb{R}^K;\mathcal{F}_T)$ and $Z\in H^2_M$, then
$Y\in S^2$ if and only if the left limit process $Y_{t-}\in H^2_\mu$.
\end{lemma}
\begin{pf}
Clearly, if $Y\in S^2$, then as $Y$ is c\`adl\`ag and adapted, and
hence progressive, $Y\in H^2_M$. For the converse, write
\begin{eqnarray*}
\sup_{t\in[0,T]}\|Y_{t}\|^2
&\leq& 2\|Q\|^2+ 4\sup_{t\in[0,T]}\biggl\|\sum_i\int
_{]t,T]}Z_u^i\,dM^i_u\biggr\|^2\\
&&{} + 4 \sup_{t\in[0,T]}\biggl\{\int_{]t,T]} \|F(\omega, u,
Y_{u-},Z_u)\|^2\,d\mu\biggr\}\\
&\leq& 2\|Q\|^2+ 4\sup_{t\in[0,T]}\biggl\|\sum_i\int
_{]t,T]}Z_u^i\,dM^i_u\biggr\|^2\\
&&{} + 8 \int_{]0,T]} \|F(\omega, u, 0,0)\|^2\,d\mu_t\\
&&{} + 8 \int_{]0,T]}
[c_t\|Y_{u-}\|^2+c\|Z_u\|^2_{M_u}]\,d\mu_t,\\
\end{eqnarray*}
and by\vspace*{1pt} the assumptions of the lemma, as $Z\in H^2_M$, and so $\sum_i\int
_{]0,t]}Z_u^i \,dM_u^i$ is a square integrable martingale, by Doob's
inequality~\cite{Jacod2003}, Theorem 1.43, this quantity is finite in expectation.
\end{pf}

The following lemma provides the key bounds on BSDE solutions, which we
shall use to prove existence and uniqueness of solutions.
\begin{lemma} \label{lemsolutionbounds}
Let\vspace*{1pt} $(Y, Z)$ and $(\bar Y, \bar Z)$ be the solutions to two BSDEs with
standard parameters $(F,Q)$ and $(\bar F,\bar Q)$. Define
%
\begin{eqnarray}\label{equppirhodefn}
\delta Y&:=& Y-\bar Y, \qquad  \delta Z := Z-\bar Z,\nonumber\\
\delta_2 f_t&:=&F(\omega, t, \bar Y_{t-}, \bar Z_t)-\bar F(\omega, t,
\bar Y_{t-}, \bar Z_t),\nonumber\\
\upsilon_t &:=& \int_{]0,t]} [(x_s^{-1} - \Delta\mu_s)(1+w_s)c_s
+x_s]\,d\mu_s,\\
\pi_t &:=& \int_{]0,t]} [(x_s^{-1} - \Delta\mu_s)(1+w_s^{-1})](1-\Delta
\upsilon_s)^{-1}\,d\mu_s,\nonumber\\
\rho_t^i &:=& \int_{]0,t]} [1-(x_s^{-1}-\Delta\mu_s)(1+w_t)c](1-\Delta
\upsilon_s)^{-1}\,d\langle M^i\rangle_t,\nonumber
\end{eqnarray}
where $c_s$ and $c$ are the Lipschitz constants of $F$, and $x_t$,
$w_t$ are any nonnegative measurable functions such that $\Delta\mu
_t\leq x_t^{-1}$ and $\Delta\upsilon_t<1$ for all $t$, and the
integrands defining $\upsilon,\pi$ and $\rho^i$ are uniformly bounded.

Then
%
\begin{eqnarray}\label{eqBSDEBound1}
&&E[\|\delta Y_t\|^2] \mathfrak{E}(\tilde\upsilon_t) +E\biggl[ \sum_i\int
_{]t,T]}\mathfrak{E}(\tilde\upsilon_{s-})\|\delta Z_s^i\|^2 \,d\rho
_s^i\biggr]\nonumber\\[-9pt]\\[-9pt]
&&\qquad \leq E[\|\delta Q\|^2] \mathfrak{E}(\tilde\upsilon_T) + \int
_{]t,T]}
E[\|\delta_2f_s\|^2]\mathfrak{E}(\tilde\upsilon_{s-})\,d\pi_s\nonumber
\end{eqnarray}
and
%
\begin{eqnarray}\label{eqBSDEBound2}
&&\int_{]0,T]} E[\|\delta Y_{t-}\|^2] \mathfrak{E}(\tilde\upsilon
_{t-})\,d\mu_t + E\biggl[ \sum_i\int_{]0,T]}\mu_s\mathfrak{E}(\tilde
\upsilon_{s-})\|\delta Z_s^i\|^2
\,d\rho_s^i\biggr]\nonumber\\[-9pt]\\[-9pt]
&&\qquad \leq\mu_T E[\|\delta Q\|^2] \mathfrak{E}(\tilde\upsilon_T) +
\int_{]0,T]}\mu_s E[\|\delta_2f_s\|^2]\mathfrak{E}(\tilde\upsilon
_{s-})\,d\pi_s.
\nonumber
\end{eqnarray}
\end{lemma}
\begin{pf}
Let $\delta F = F(\omega, t, Y_{t-}, Z_t)-\bar F(\omega, t, \bar
Y_{t-}, \bar Z_{t})$. By application of the differentiation rule for
stochastic integrals, we have
%
\begin{eqnarray} \label{eqexpanddeltaY}\quad
d[\|\delta Y_t\|^2] &=& -2(\delta Y_{t-})^*(\delta F_t) \,d\mu_t + 2 \sum
_i(\delta Y_{t-})^*(\delta Z_t^i) \,dM^i_t\nonumber\\[-2pt]
&&{} + \sum_{i,j}(\delta Z_t^i)^*(\delta Z_t^j) \,d[ M^i, M^j]_t -2
(\delta F_t) (\Delta\mu_t) \sum_i (\delta Z_t^i) \Delta M^i_t\\[-2pt]
&&{} +\|\delta F_t\|^2 (\Delta\mu_t)^2.\nonumber
\end{eqnarray}
As $\delta Y\in S^2$, by the BDG inequality it is clear that $\int
_{]0,t]}\sum_i(\delta Y_{s-})^* (\delta Z_s^i) \,dM^i_s$ is a martingale.
Similarly the process
\[
\sum_{s\in]0,t]} \biggl[(\delta F_s) (\Delta\mu_s) \sum_i (\delta
Z_s^i) \Delta M^i_s \biggr]
\]
is a countable sum of integrable martingale differences and so is also
a martingale. Also, $\delta Z\in H^2_M$ and so, by orthogonality of the $M^i$,
\[
\sum_{i,j}(\delta Z_t^i)^*(\delta Z_t^j) \,d[ M^i, M^j]_t - \sum_i\|
\delta Z_t^i\|^2\,d\langle M^i\rangle_t
\]
is a martingale.

For any $A\in\mathcal{B}([0,T])$, integrating on $A$ and taking an
expectation through (\ref{eqexpanddeltaY}) then yields
\begin{eqnarray*}
\int_A dE[\|\delta Y_t\|^2]
&=& -2\int_AE[(\delta Y_{t-})^*(\delta F_t)] \,d\mu_t + E\biggl[\sum_i\int
_A\|\delta Z_t^i\|^2 \,d\langle M^i\rangle_t\biggr] \\[-2pt]
&&{} + \sum_{t\in A}E[\|\delta F_t\|^2] (\Delta\mu_t)^2.
\end{eqnarray*}
Using the fact that $(\Delta\mu_t)^2 = (\Delta\mu_t) (d\mu_t)$ and
that for any $x\geq 0$, any $a,b\in\mathbb{R}$, $\pm2ab \leq
xa^2+x^{-1} b^2$, we have, for any measurable function $x_t\geq0$,
%
\begin{eqnarray}\label{eqexpand1}\quad
\int_A dE[\|\delta Y_t\|^2] &\geq& -\int_Ax_t E[\|\delta Y_{t-}\|^2] \,d\mu
_t -\int_Ax_t^{-1} E[\|\delta F_t\|^2] \,d\mu_t \nonumber\\[-2pt]
&&{} + E\biggl[\sum_i\int_A\|\delta Z_t^i\|^2 \,d\langle M^i\rangle
_t\biggr] + \int_AE[\|\delta F_t\|^2] (\Delta\mu_t) \,d\mu_t\\[-2pt]
&=&-\int_Ax_t E[\|\delta Y_{t-}\|^2] \,d\mu_t -\int_A(x_t^{-1}-\Delta\mu
_t) E[\|\delta F_t\|^2] \,d\mu_t \nonumber\\[-2pt]
&&{} + E\biggl[\sum_i\int_A\|\delta Z_t^i\|^2 \,d\langle M^i\rangle
_t\biggr].
\nonumber
\end{eqnarray}

We now note that, for any measurable $w_t\geq0$, as $(a+b)^2 \leq
(1+w) a^2 + (1+w^{-1})b^2$ for all $w\geq0$,
\begin{eqnarray*}
\|\delta F_t\|^2 &\leq& (1+w_t)\|F(\omega, t, Y_{t-}, Z_t) - F(\omega,
t, \bar Y_{t-}, \bar Z_t)\|^2\\
&&{}  + (1+w_t^{-1})\|F(\omega, t,
\bar Y_{t-}, \bar Z_t)-\bar F(\omega, t, \bar Y_{t-}, \bar Z_t)\|^2\\
&\leq& (1+w_t)c_t \|\delta Y_{t-}\|^2 + (1+w_t) c\|\delta Z_t\|^2_{M_t}
+ (1+w_t^{-1})\|\delta_2 f_t\|^2.
\end{eqnarray*}
Hence, as $x_t^{-1}-\Delta\mu_t\geq0$,
%
\begin{eqnarray}\label{eqexpand2}
&&\int_A (x_t^{-1}-\Delta\mu_t)E[\|\delta F_t\|^2]\,d\mu_t \nonumber\\[-2pt]
&&\qquad\leq\int_A(x_t^{-1}-\Delta\mu_t)(1+w_t)c_t E[\|\delta Y_{t-}\|^2]\,d\mu
_t \nonumber\\[-2pt]
&&\qquad\quad{} + \int_A(x_t^{-1}-\Delta\mu_t)(1+w_t) cE[\|\delta Z_t\|^2_{M_t}]\,d\mu
_t \nonumber\\[-2pt]
&&\qquad\quad{} + \int_A(x_t^{-1}-\Delta\mu_t)(1+w_t^{-1})E[\|\delta_2 f_t\|
^2]\,d\mu_t \\[-2pt]
&&\qquad\leq\int_A(x_t^{-1}-\Delta\mu_t)(1+w_t)c_t E[\|\delta Y_{t-}\|^2]\,d\mu
_t \nonumber\\[-2pt]
&&\qquad\quad{} +
E\biggl[\sum_i\int_A(x_t^{-1}-\Delta\mu_t)(1+w_t) c\|\delta Z_t^i\|^2
\,d\langle
M^i\rangle_t\biggr] \nonumber\\[-2pt]
&&\qquad\quad{} + \int_A(x_t^{-1}-\Delta\mu_t)(1+w_t^{-1})E[\|\delta_2 f_t\|
^2]\,d\mu_t.
\nonumber
\end{eqnarray}
Combining~(\ref{eqexpand1}) and~(\ref{eqexpand2}) gives
\begin{eqnarray*}
\int_A dE[\|\delta Y_t\|^2] &\geq& -\int_A E[\|\delta Y_{t-}\|^2]
\,d\upsilon_t + E\biggl[\sum_i\int_A (1-\Delta\upsilon_t)\|\delta Z_t^i\|
^2 \,d\rho_t^i\biggr] \\[-2pt]
&&{} - \int_A E[\|\delta_2f_t\|^2](1-\Delta
\upsilon_t)\,d\pi_t.
\end{eqnarray*}

Let $\phi$ be the signed measure on $\mathcal{B}([0,T])$ defined by
\[
\phi(A) = E\biggl[\sum_i\int_A (1-\Delta\upsilon_t)\|\delta Z_t^i\|^2
\,d\rho_t^i\biggr] - \int_A E[\|\delta_2f_t\|^2](1-\Delta\upsilon_t)\,d\pi_t.
\]
As $d\pi/d\mu$ is bounded, $d\rho^i/d\langle M^i\rangle$ is bounded, $\|
\delta_2f\|^2$ is $\mu$-integrable and \mbox{$\delta Z_t \in H^2_M$}, it
follows that $\phi(A)$ is bounded. We see then that $\phi$ is a signed
Stieltjes measure, and we equate it with its distribution function $\phi
_t:=\phi([0,t])$.

Therefore, as $\Delta\upsilon_t <1$, $\Delta\mu-x^{-1}\leq0$, an
application of Lemma~\ref{lemintegratingfactorbound} yields
\begin{eqnarray*}
\int_A d[E[\|\delta Y_t\|^2] \mathfrak{E}(\tilde\upsilon_t)]
&\geq&\int_A(1-\Delta\upsilon_t)^{-1}\mathfrak{E}(\tilde\upsilon
_{t-})\,d\phi_t \\
&=& E\biggl[\sum_i\int_A\mathfrak{E}(\tilde\upsilon_{t-})\|\delta Z_t^i\|
^2 \,d\rho_t^i\biggr]\\
&&{}- \int_A\mathfrak{E}(\tilde\upsilon_{t-})E[\|\delta
_2f_t\|^2]\,d\pi_t.
\end{eqnarray*}
For $A= \ ]t,T]$, it follows that
\begin{eqnarray*}
&&E[\|\delta Y_t\|^2] \mathfrak{E}(\tilde\upsilon_t) +E\biggl[ \sum_i\int
_{]t,T]}\mathfrak{E}(\tilde\upsilon_{s-})\|\delta Z_s^i\|^2 \,d\rho
_s^i\biggr]\\
&&\qquad \leq E[\|\delta Q\|^2] \mathfrak{E}(\tilde\upsilon_T) + \int
_{]t,T]} E[\|\delta_2f_s\|^2]\mathfrak{E}(\tilde\upsilon_{s-})\,d\pi_s,
\end{eqnarray*}
which is the desired inequality~(\ref{eqBSDEBound1}). Taking a
left-limit in $t$ gives, by the dominated convergence theorem,
\begin{eqnarray*}
&&E[\|\delta Y_{t-}\|^2] \mathfrak{E}(\tilde\upsilon_{t-}) + E\biggl[
\sum_i\int_{[t,T]}\mathfrak{E}(\tilde\upsilon_{s-})\|\delta Z_s^i\|^2
\,d\rho_s^i\biggr]\\
&&\qquad \leq E[\|\delta Q\|^2] \mathfrak{E}(\tilde\upsilon_T) + \int
_{[t,T]} E[\|\delta_2f_s\|^2]\mathfrak{E}(\tilde\upsilon_{s-})\,d\pi_s,
\end{eqnarray*}
and so by integration and Fubini's theorem, we have that
\begin{eqnarray*}
&&\int_{]0,T]} E[\|\delta Y_{t-}\|^2] \mathfrak{E}(\tilde\upsilon
_{t-})\,d \mu_t + E\biggl[ \sum_i\int_{]0,T]}\mu_s\mathfrak{E}(\tilde
\upsilon_{s-})\|\delta Z_s^i\|^2 \,d\rho_s^i\biggr]\\
&&\qquad \leq\mu_T E[\|\delta Q\|^2] \mathfrak{E}(\tilde\upsilon_T) +
\int_{]0,T]}\mu_s E[\|\delta_2f_s\|^2]\mathfrak{E}(\tilde\upsilon
_{s-})\,d\pi_s.
\end{eqnarray*}
\upqed\end{pf}
\begin{lemma} \label{lemsimplestBSDE}
Let $F\dvtx\Omega\times[0,T] \to\mathbb{R}^K$ be a predictable
progressively measurable function such that
\[
E\biggl[\int_{]0,T]} \|F(\omega, t)\|^2 \,d\mu\biggr] <+\infty.
\]
Then the BSDE
\[
Y_t - \int_{]t,T]} F(\omega, u) \,d\mu+ \sum_i \int_{]t,T]} Z^i_u\,dM^i_u= Q
\]
has a unique solution in $S^2\times H^2_M$ for any $Q\in L^2(\mathbb
{R}^K;\mathcal{F}_T)$. (Note here that~$F$ does not depend on $Y$ or $Z$.)
\end{lemma}
\begin{pf}
Using Theorem~\ref{thmmartrep}, we first construct the processes $Z^i$
which give a representation of the square integrable martingale
\[
\sum_i \int_{]0,t]} Z^i_u \,dM^i_u = E\biggl[ Q + \int
_{]0,T]}F(\omega, u) \,d\mu\Big|\mathcal{F}_t\biggr].
\]
This can clearly be done componentwise, and so we obtain a unique
process $Z\in H^2_M$, that is, $Z^i_s(\omega) \in\mathbb{R}^K$.
It follows that
\begin{eqnarray*}
\sum_i \int_{]t,T]}Z^i_u \,dM^i_u &=& Q + \int_{]0,T]}F(\omega, u) \,d\mu-
E\biggl[ Q + \int_{]0,T]}F(\omega, u) \,d\mu\Big|\mathcal
{F}_t\biggr]\\
&=& Q + \int_{]t,T]}F(\omega, u) \,d\mu- E\biggl[ Q + \int
_{]t,T]}F(\omega, u) \,d\mu\Big|\mathcal{F}_t\biggr],
\end{eqnarray*}
and so there is an adapted process
%
\begin{eqnarray}\label{eqsimpbsdesoln}
Y_t:\!&=& E\biggl[ Q + \int_{]t,T]}F(\omega, u) \,d\mu\Big|\mathcal
{F}_t\biggr] \nonumber\\[-8pt]\\[-8pt]
&=& Q + \int_{]t,T]}F(\omega, u) \,d\mu-\sum_i \int
_{]t,T]}Z^i_u \,dM^i_u,\nonumber
\end{eqnarray}
which satisfies the BSDE. By uniqueness of the right-hand side of (\ref
{eqsimpbsdesoln}), this process is unique up to indistinguishability
and hence in $S^2$.
\end{pf}
\begin{lemma}\label{lemboundepsilon}
Let $\nu\dvtx[0,T]\to\mathbb{R}$ be a nondecreasing c\`adl\`ag function of
finite variation and $c_{(\cdot)}\dvtx[0,T]\to\mathbb{R}$ be a nonnegative
bounded measurable function. Then $c_t\Delta\nu_t = \sup_{s\in[0,T]}\{
c_s\Delta\nu_s\}$ for some $t$; that is $c_t\Delta\nu_t$ attains its
maximum. Consequently, if, for some $k\in\mathbb{R}$, $c_t\Delta\nu_t
<k$ for all $t$, then there exists an $\varepsilon>0$ such that $c_t\Delta
\nu_t<k-\varepsilon$ for all $t$.
\end{lemma}
\begin{pf}
If $c_t\Delta\nu_t\equiv0$, then the result is trivial. Let $c$ be the upper
bound of $c_{(\cdot)}$. As $\nu$ is right-continuous, it has at most countably
many jumps. Then, as $\nu$ is nondecreasing, $\sum_t c_t\Delta\nu_t\leq
c(\sum_t \Delta\nu_t)\leq c\nu_T<\infty$. Therefore,
$c_t\Delta
\nu_t$ is a summable sequence, and hence has finitely many values
greater than
or equal to $\delta$, for any $\delta>0$. Let $\delta\in\ ]0,
c_t\Delta\nu_t]$
for some $t$, and so $\{c_t \Delta\nu_t\dvtx c_t \Delta\nu_t\geq\delta\}$
is a
finite nonempty set, and therefore has a maximum.

Now suppose $c_t\Delta\nu_t<k$ for all $t$. Let $t^*$ be the value at
which $c_t\Delta\nu_t$ attains its maximum, hence $c_{t^*}\Delta\nu
_{t^*}<k$. For any $\varepsilon< k-c_{t^*}\Delta\nu_{t^*}$ the result
then holds.
\end{pf}
\begin{pf*}{Proof of Theorem~\ref{thmBSDEExist}}
We consider constructing a sequence of approximations in the usual way. For
a BSDE with driver $F$ and terminal condition $Q$, we fix an initial
approximation $(Y^0, Z^{(0)})\in S^2\times H^2_M$. (Note that we denote
by $Z^{(n)}$ the $n$th approximation of the infinite-dimensional
process $Z$, to
distinguish it from~$Z^i$, the $i$th component of $Z$.) We shall
first allow the $Z$ component of the solution to converge, then allow
the $Y$
component to do likewise. This two-stage approach is needed due to the
difference in the Lipschitz coefficients of $F$ with respect to $Y$ and
$Z$. We
shall assume, without loss of generality, that the Lipschitz
coefficient of $F$
(with respect to $Z$) satisfies $c>0$ uniformly.

\subsection*{Step 1: BSDEs where the driver has $Y$ fixed}

To construct the $Z$ solutions, we first fix some c\`adl\`ag process
$\tilde Y \in H^2_\mu$. We wish to define a sequence of approximations
of solutions to the BSDE with driver $F(\cdot, \cdot, \tilde Y_{\cdot
-}, \cdot)$.

For any approximation $Z^{(n)}$, we fix the driver $F^n(\omega,
t)=F(\omega, t, \tilde Y_{t-}, Z^{(n)}_t)$. Using Lemma \ref
{lemsimplestBSDE}, we obtain a new approximation $(Y^{n+1},
Z^{(n+1)})$. We shall show that the induced map $Z^{(n)}\mapsto
Z^{(n+1)}$ is a contraction, and hence that a unique limit exists.

Suppose\vspace*{1pt} at the $n$th stage we have two approximations $(Y^{n,1},
Z^{(n,1)})$ and $(Y^{n,2}, Z^{(n,2)})$ of the solution of a BSDE with
terminal value $Q$ and driver $F(\cdot, \cdot, \tilde Y_{\cdot-}, \cdot
)$. We can hence construct new approximations $(Y^{n+1,1},
Z^{(n+1,1)})$ and $(Y^{n+1,2}, Z^{(n+1,2)})$. We consider the difference
\[
\bigl(\delta Y^{n+1}, \delta Z^{(n+1)}\bigr) = \bigl(Y^{n+1,1}-Y^{n+1,2},
Z^{(n+1,1)}-Z^{(n+1,2)}\bigr).
\]
Note that $(Y^{n+1,1}, Z^{(n+1,1)})$ comes from a BSDE with driver
$F(\cdot, \cdot, \tilde Y_{\cdot-}, Z^{(n,1)}_\cdot)$ which does not
depend on the solutions $(Y^{n+1,1}, Z^{(n+1,1)})$. Hence, for
appropriate functions $x_\cdot$ and $w_\cdot$, the differences $(\delta
Y^{n+1}, \delta Z^{(n+1)})$ satisfy our estimate~(\ref{eqBSDEBound1}), with
\[
\delta_2f_s = F\bigl(\omega, s, \tilde Y_{s-}, Z^{(n,1)}_s\bigr) - F\bigl(\omega, s,
\tilde Y_{s-}, Z^{(n,2)}_s\bigr)
\]
and $\delta Q = 0$, and
when defining $\upsilon$ and $\rho^i$ in~(\ref{equppirhodefn}) we
can take
$c_s=c=0$.

We take the values $w_t\!=\!1,
x_t^{-1}\! =\! \frac{1}{4c} \!+\! \Delta\mu_t$, and so we see that
\mbox{$\Delta\mu_t\!-\!x_t^{-1}\! \leq\!0$},
\[
\upsilon_t = \int_{]0,t]} x_s\,d\mu_s = \int_{]0,t]} \frac{4c}{1+4c\Delta
\mu_s}\,d\mu_s \leq4c\mu_t
\]
is nondecreasing and bounded (and hence of finite variation) and
\[
\Delta\upsilon_t = \frac{4c\Delta\mu_t}{1+4c\Delta\mu_t} \leq1-\frac
{1}{1+4c}<1.
\]
It follows that the integrands in~(\ref{equppirhodefn}) are bounded,
our estimate~(\ref{eqBSDEBound1}) holds and $\mathfrak{E}(\tilde
\upsilon_{s-})(1-\Delta\upsilon_s)^{-1}$ is strictly positive and
bounded. Hence
\[
Z\mapsto E\biggl[\sum_i\int_{]0,T]} \|Z_s^i\|^2 \mathfrak{E}(\tilde
\upsilon_{s-}) (1-\Delta\upsilon_s)^{-1}\,d\langle M^i\rangle_s\biggr]
\]
is an equivalent norm on $H^2_{M}$.

As we can take $c=0$ in~(\ref{equppirhodefn}), we have the simplification
\[
d\rho^i_t=(1-\Delta\upsilon_t)^{-1}\,d\langle
M^i\rangle_t=(1-x_t\Delta\mu_t)^{-1}\,d\langle M^i\rangle_t.
\]
From
(\ref{eqBSDEBound1}) we obtain
\begin{eqnarray*}
&&E\biggl[ \sum_i\int_{]t,T]}\bigl\|\bigl(\delta Z^{(n+1)}\bigr)_s^i\bigr\|^2 \mathfrak
{E}(\tilde
\upsilon_{s-})(1-\Delta\upsilon_s)^{-1}\,d\langle M^i\rangle_t\biggr] \\
&&\qquad \leq \int_{]t,T]} E[\|\delta_2f_s\|^2][(x_s^{-1} - \Delta\mu
_s)(1+w_s^{-1})]\mathfrak{E}(\tilde\upsilon_{s-}) (1-\Delta\upsilon
_s)^{-1}\,d\mu_s\\
&&\qquad = \int_{]t,T]} E[\|\delta_2f_s\|^2]\biggl[\frac{1}{2c}
\biggr]\mathfrak{E}(\tilde\upsilon_{s-}) (1-\Delta\upsilon_s)^{-1}\,d\mu_s.
\end{eqnarray*}

By the Lipschitz continuity of the original driver, we have
\[
E[\|\delta_2f_s\|^2] \leq cE\bigl[\bigl\|\delta Z^{(n)}_s\bigr\|^2_{M_s}\bigr],
\]
and so, for our chosen values of $w_t$ and $x_t$, using inequality (\ref
{eqisometry}),
\begin{eqnarray*}
&&E\biggl[ \sum_i\int_{]0,T]}\bigl\|\bigl(\delta Z^{(n+1)}\bigr)_s^i\bigr\|^2
\mathfrak {E}(\tilde
\upsilon_{s-})(1-\Delta\upsilon_s)^{-1}\,d\langle M^i\rangle_t\biggr] \\
&&\qquad \leq \frac{1}{2}\int_{]0,T]}E\bigl[\bigl\|\delta Z^{(n)}_s\bigr\|^2_{M_s}\bigr]
\mathfrak{E}(\tilde\upsilon_{s-})(1-\Delta\upsilon_s)^{-1}\,d\mu_s\\
&&\qquad \leq \frac{1}{2}E\biggl[ \sum_i\int_{]0,T]}\bigl\|\bigl(\delta
Z^{(n)}\bigr)_s^i\bigr\|^2
\mathfrak{E}(\tilde\upsilon_{s-})(1-\Delta\upsilon_s)^{-1}\,d\langle
M^i\rangle_t\biggr].
\end{eqnarray*}

By completeness, the contraction mapping principle gives the existence
of a unique limit $Z\in H^2_M$ solving the BSDE with driver $F(\cdot,
\cdot, \tilde Y_{\cdot-}, \cdot)$ and terminal value $Q$. (The solution
$Y$ process can, of course, be found using Lemma \ref
{lemsimplestBSDE}, fixing the $Z$ process at the constructed limit.)

\subsection*{Step 2: BSDEs with general drivers}

We now construct a convergent sequence of approximations in $Y$ for a
general driver. Consider the Lipschitz bounds of the original driver
$F$. Without
loss of generality, we assume that $c>0$ uniformly. As $c_s\Delta\mu_s<1$,
$\mu$ is nondecreasing and of finite variation, and $c_s$ is bounded,
Lem\-ma~\ref{lemboundepsilon} yields a fixed $\varepsilon>0$ such that $c_s\Delta
\mu_s <
1-\varepsilon$.

Let\vspace*{-2pt}
\begin{eqnarray*}
x_t^{-1} &=& \frac{1}{c(1+2\varepsilon^{-1})} +\Delta\mu_t,\\[-3pt]
w_t^{-1} &=& \frac{\varepsilon}{2}+\frac{\varepsilon^2}{8-4\varepsilon}.\vspace*{-2pt}
\end{eqnarray*}

As\vspace*{-2pt}
\[
x_t^{-1} -\Delta\mu_t = \frac{1}{c(1+2\varepsilon^{-1})} < \frac{1}{c(1+w_t)},\vspace*{-2pt}
\]
it is clear that
\[
\frac{d\rho_t^i}{d\langle M^i\rangle_t} = [1-(x_t^{-1} - \Delta\mu
_t)(1+w_t)c] (1-x_t\Delta\mu_t)^{-1}> 0,\vspace*{-2pt}
\]
so $\rho^i$ is a nonnegative measure for each $i$.

For any terminal value $Q$, consider an approximation $Y^n \in S^2$. We
can then
construct a solution $(Y^{n+1}, Z^{(n+1)})$ to the BSDE with driver
$F^n(\omega, t, z)=F(\omega,t, Y^n_{t-}, z)$, using the above result.
Again we shall show that $Y^n\mapsto Y^{n+1}$ is a contraction, and
hence that a unique limit exists.

As above, we consider the sequence of differences $(\delta Y^n, \delta
Z^{(n)})$ from two initial approximations. As $Y^{n+1,1}$ is defined
using\vspace*{1pt} the driver $F^n=F(\omega,t,Y^{n,1}_{t-},\break  z)$, which does not
depend on $Y^{n+1,1}$, we can take $c_s=0$
when defining~$\rho^i$ in~(\ref{equppirhodefn}).

Hence, for our chosen values of $x_t$ and $w_t$, we can
again easily verify that the integrands in~(\ref{equppirhodefn}) are bounded,
and the resulting $\upsilon$ is nonnegative, bounded and $\Delta\upsilon
<1$. It
follows that $\mathfrak{E}(\tilde\upsilon_s)$ is strictly positive and bounded.

Considering the difference of any two approximations $\delta Y^n$, by
the Lipschitz continuity of the original driver, we have
\[
E[\|\delta_2f_s\|^2] \leq c_sE[\|\delta Y^{n}_{s-}\|^2],\vspace*{-2pt}
\]
so, as $\rho^i$ is a family of nonnegative measures, our estimate (\ref
{eqBSDEBound2}) gives
\begin{eqnarray*}
&&\int_{]0,T]} E[\|\delta Y_{t-}^{n+1}\|^2] \mathfrak{E}(\tilde\upsilon
_{t-})\,d\mu_t \\[-3pt]
&&\qquad \leq\int_{]0,T]}\mu_s E[\|\delta_2f_s\|^2]\mathfrak{E}(\tilde
\upsilon_{s-})\,d\pi_s\\[-3pt]
&&\qquad \leq\int_{]0,T]}E[\|\delta Y^{n}_{s-}\|^2]\mathfrak{E}(\tilde
\upsilon_{s-})\mu_s c_s[(x_s^{-1} - \Delta\mu
_s)(1+w_s^{-1})]\\[-3pt]
&&\hphantom{\int_{]0,T]}}\qquad\quad{}\times(1-x_s\Delta\mu_s)^{-1}\,d\mu_s.\vspace*{-2pt}
\end{eqnarray*}

By construction we have
\begin{eqnarray*}
&&\mu_s c_s[(x_s^{-1} - \Delta\mu_s)(1+w_s^{-1})](1-x_s\Delta\mu
_s)^{-1}\\[-2pt]
&&\qquad =\mu_s c_sx_s^{-1}(1+w_s^{-1})\\[-2pt]
&&\qquad = \mu_s c_s\biggl(\frac{1}{c(1+2\varepsilon^{-1})} + \Delta\mu
_s\biggr)(1+w_s^{-1})\\[-2pt]
&&\qquad \leq\mu_s\biggl(\frac{c_s}{c(1+2\varepsilon^{-1})} + 1-\varepsilon
\biggr)\biggl(1+\frac{\varepsilon}{2}+\frac{\varepsilon^2}{8-4\varepsilon}
\biggr)\\[-2pt]
&&\qquad \leq\biggl(1-\frac{\varepsilon}{2}\biggr)\biggl(1+\frac{\varepsilon
}{2}+\frac{\varepsilon^2}{8-4\varepsilon}\biggr)\\[-2pt]
&&\qquad = 1-\frac{\varepsilon^2}{8},
\end{eqnarray*}
where the fifth line is because $\mu_s\leq\mu_T\leq1$ and
\[
\frac{c_s}{c} \leq1 < 1+\frac{\varepsilon}{2} = \frac{\varepsilon}{2}
(1+2\varepsilon^{-1}).
\]

We then have
\begin{eqnarray*}
\int_{]0,T]} E[\|\delta Y_{t-}^{n+1}\|^2] \mathfrak{E}(\tilde\upsilon
_{t-})\,d\mu_t \leq\biggl(1-\frac{\varepsilon^2}{8}\biggr)\int_{]0,T]}E[\|
\delta Y^{n}_{s-}\|^2]\mathfrak{E}(\tilde\upsilon_{s-})\,d\mu_s.
\end{eqnarray*}
As $\mathfrak{E}(\tilde\upsilon_{s-})$ is strictly positive and
bounded, $\int_{]0,T]} E[\|\cdot\|^2]\mathfrak{E}(\tilde\upsilon
_{s-})\,d\mu_s$ is an equivalent norm on $H^2_\mu$. By completeness, the
contraction mapping principle gives the existence of a limit $Y^\infty
_t = \lim_{n\to\infty} Y^n_{t-}$, which is unique in $H^2_\mu$. We also
have the existence of a limit $Z$, as $d\rho^i_t/d\langle M^i\rangle_t$
is strictly positive, and from~(\ref{eqBSDEBound2}),
\begin{eqnarray*}
&&\lim_{n\to\infty} E\biggl[ \sum_i\int_{]0,T]}\mu_s\mathfrak{E}(\tilde
\upsilon_{s-})\bigl\|\bigl(\delta Z^{(n)}\bigr)_s^i\bigr\|^2 \,d\rho_s^i\biggr]\\[-2pt]
&&\qquad \leq
\lim_{n\to\infty}\int_{]0,T]}\mu_s E[\|\delta_2f_s^n\|^2]\mathfrak
{E}(\tilde
\upsilon_{s-})\,d\pi_s =0;
\end{eqnarray*}
that is, $\delta Z^{(n)}$ also converges to zero in $H^2_M$.

We take the right limits of a left-continuous version of the process
$Y^\infty$, namely
\begin{eqnarray*}
Y_t &=& E\biggl[Q+\int_{]t,T]} F(\omega, s, Y^{\infty}_{s}, Z_s)\,d\mu
_s\Big|\mathcal{F}_t\biggr] \\[-2pt]
&=& E\biggl[Q+\int_{]t,T]} F(\omega,
s, Y_{s-}, Z_s)\,d\mu_s\Big|\mathcal{F}_t\biggr].
\end{eqnarray*}

By Lemma~\ref{lemYquadbound}, $Y\in S^2$ and by Lemma \ref
{lemH2S2equiv} it is unique in $S^2$.
This pair $(Y,Z)$ will solve the BSDE with driver $F$ and terminal
value $Q$. This limit is unique for $Z\in H^2_M$, as can by seen by
fixing $Y$ and using our earlier result.
\end{pf*}
\begin{remark}
In discrete time, we have shown in~\cite{Cohen2008c} that a necessary
and sufficient condition for the existence\vadjust{\goodbreak} of a solution to the
discrete BSDE is that $F$ is invariant with respect to equivalent $Z$
in \mbox{$\|\cdot\|_{M_t}$} norm, and that $y\to y-F(\omega, t, y, z)$ is a
bijection in $y$ for all $z, t$ and almost all $\omega$. The
requirement that $F$ is firmly Lipschitz is sufficient, but not
necessary, to guarantee that these conditions hold.
\end{remark}

\section{Existence of BSDE solutions: General results}

We now wish to extend our above solution to allow $\mu$ to be any
Stieltjes measure, by relaxing the condition that $\mu_T\leq1$. In so
doing, we shall also weaken slightly the firm Lipschitz requirement.
\begin{lemma}\label{lemsplitupinterval}
Let $\nu$ be a nonnegative Stieltjes measure with $\Delta\nu<1$. Then
there exists an $\eta>0$ and a finite sequence $\{0=t_0< t_1<\cdots<t_B=T\}
$ such that $\nu(]t_i, t_{i+1}])\leq1-\eta$ for all $i$.
\end{lemma}
\begin{pf}
By Lemma~\ref{lemboundepsilon} with $c_t\equiv1$, there exists an
$\eta>0$ with $\Delta\nu<1-\eta$. Let $t_K=T$ for some large $K$.
Define recursively for integers $j<K$, $t_j = \sup\{t\dvtx\nu_t< \nu
_{t_{j+1}} -1 + \eta\}\vee0$. By right continuity, $\nu(]t_j,
t_{j+1}]) = \nu_{t_{j+1}}-\nu_{t_j} \leq1-\eta$. For any $j$, it is
also easy to show that $\nu_{t_{j+2}}-\nu_{t_j}>1-\eta$. Hence, as $\nu
_T$ is finite, the sequence $t_j$ has only finitely many nonzero terms.
Let $k = \max\{j\dvtx t_k=0\}$, let $B=K-k$ and rescale the index of our
sequence accordingly. We then have a sequence with the desired properties.
\end{pf}
\begin{theorem} \label{thmBSDEExist2}
Let $\mu$ be any deterministic Stieltjes measure assigning positive
measure to every open interval. (Note\vspace*{1pt} \mbox{$\|\cdot\|_M$} is still well
defined in relation to $\mu$.) Let $F\dvtx\Omega\times[0,T]\times\mathbb
{R}^K\times\mathbb{R}^{K\times\infty}\to\mathbb{R}^K$ be a predictable,
progressively measurable function such that:
\begin{itemize}
\item$E[\int_{]0,T]} \|F(\omega, t, 0,0)\|^2\,d\mu_t]<+\infty$.
\item There exists a quadratic firm Lipschitz bound on $F$, that is, a
measurable deterministic function $c_t$ uniformly bounded by some $c\in
\mathbb{R}$, such that, for any $y_t, y_t'\in\mathbb{R}^K$, $z_t,
z'_t\in\mathbb{R}^{K\times\infty}$,
\begin{eqnarray*}
&&\|F(\omega, t, y_t, z_t) - F(\omega, t, y'_t, z'_t)\|^2\\
&&\qquad\leq c_t \|
y_t-y'_t\|^2+ c\|z_t-z'_t\|^2_{M_t},\qquad \,d\mu\times d\mathbb{P}\mbox{-a.s.}
\end{eqnarray*}
and
\[
c_t(\Delta\mu_t)^2<1.
\]
Note that the variable bound $c_t$ need only apply to the behavior of
$F$ with respect to $y$.
\end{itemize}
A function\vspace*{1pt} satisfying these conditions will be called \textit{standard}.
Then for any $Q\in L^2(\mathbb{R}^K;\mathcal{F}_T)$, the BSDE (\ref
{eqBSDE}) with driver $F$ has a unique solution $(Y,Z)\in S^2\times
H^2_M$. ($S^2$ and $H^2_M$ are defined in Definition~\ref{defnprocessspaces}.)
\end{theorem}
\begin{pf}
We assume, without loss of generality, that $c\geq1$. By
Lemma~\ref{lemboundepsilon}, as $(\mu_t)^2$ is a nondecreasing c\`adl\`ag
function of
finite variationz\vadjust{\goodbreak} and $c_t(\Delta\mu_t)^2<1$ for all $t$, there exists an
$\varepsilon>0$ such that $c_t(\Delta\mu_t)^2 \leq1-\varepsilon$. Let
\[
\nu_t=\int_{]0,t]} \frac{2(1+\varepsilon^{-1})c}{\varepsilon+2(1+\varepsilon
^{-1})c\Delta\mu_t} \,d\mu_t =:\int_{]0,t]}\lambda_t^{-1} \,d\mu_t.
\]
Then $\nu\sim\mu$, and $\Delta\nu_t = \lambda_t^{-1} \Delta\mu_t <1$.
As $\nu_t$ is right continuous, deterministic and has no jumps of size
equal to or greater than one, by Lemma~\ref{lemsplitupinterval} there
exists a finite sequence $\{t_0=0< t_1<\cdots<t_B=T\}$ such that $\nu
(]t_j, t_{j+1}])\leq1$ for all $j$.

We now note that, omitting the $\omega$ and $t$ arguments, our BSDE
(\ref{eqBSDE}) can be written
%
\begin{equation}\label{eqBSDEdistorted}
Q = Y_t - \int_{]t,T]} \lambda_uF(Y_{u-}, Z_u) \,d\nu_u + \sum
_{i=1}^\infty\int_{]t,T]}Z^i_u\,dM^i_u,
\end{equation}
which is a BSDE in $\nu$ with Lipschitz property
\[
\|\lambda_tF(y_t, z_t) - \lambda_tF(y'_t, z'_t)\|^2\leq\lambda_t^2c_t
\|y_t-y'_t\|^2+ \lambda_t^2c\|z_t-z'_t\|^2_{M_t},\qquad   \,d\nu\times
d\mathbb{P}\mbox{-a.s.}
\]

We write
\[
\bar c = \sup_t \{\lambda_t^2 c\}\leq\biggl(\frac{\varepsilon}{2(1+\varepsilon
^{-1})c}+\mu_T\biggr)^2c<\infty
\]
and $\bar c_t = \lambda_t^2 c_t$. Note that as $\varepsilon<1$, $c_t/c<1$,
%
\begin{eqnarray}\label{eqcbound}
\bar c_t \Delta\nu_t &=& \biggl(\frac{\varepsilon+2(1+\varepsilon^{-1})c\Delta
\mu_t}{2(1+\varepsilon^{-1})c}\biggr)^2c_t \Delta\nu_t\nonumber\\
&\leq&\biggl(\frac{\varepsilon}{2(1+\varepsilon^{-1})c}+\Delta\mu_t
\biggr)^2c_t\nonumber\\
&\leq&(1+\varepsilon^{-1})\frac{\varepsilon^2c_t}{4(1+\varepsilon
^{-1})^2c^2}+(1+\varepsilon)c_t(\Delta\mu_t)^2\nonumber\\[-8pt]\\[-8pt]
&\leq&\frac{\varepsilon^2}{4}+(1+\varepsilon)(1-\varepsilon)\nonumber\\
&\leq&1-\frac{3\varepsilon^2}{4}\nonumber\\
&<&1.
\nonumber
\end{eqnarray}
Finally, we define the measures
\[
\nu^k_t = \int_{]0,t\wedge t_{k+1}]}\biggl(\frac{\eta}{\nu_{t_k}} +
\biggl(1-\frac{\eta}{\nu_{t_k}}\biggr)I_{t>t_k}\biggr) \,d\nu_t.
\]
It is easy then to show that $\nu^k_{t_{k+1}}\leq1$ for all $k$.
Furthermore, as
\[
\frac{d\nu^k}{d\nu} = \frac{\eta}{\nu_{t_k}} I_{t<t_k} + I_{t\in[t_k,
t_{k+1}]}>0,
\]
we see $\nu^k$ assigns positive measure to every interval in
$]0,t_{k+1}]$. Hence, on $]0,t_{k+1}]$, $\nu^k$ is a measure of the
type considered in Theorem~\ref{thmBSDEExist}. Also, $\Delta\nu^k_t
\leq\Delta\nu_t < 1$, and $\nu^k$ agrees with $\nu$ for all subsets of
$]t_k, t_{k+1}]$.

We now consider the sequence of BSDEs
%
\begin{equation} \label{eqBSDEiterating}
Y^{k+1}_{t_{k+1}} = Y^k_t - \int_{]t,t_{k+1}]} \lambda_tF(Y^k_{u-},
Z^k_u) \,d\nu_u^k + \sum_{i=1}^\infty\int_{]t,t_{k+1}]}(Z^k)^i_u\,dM^i_u
\end{equation}
with $Y^{B}_T = Q$. For each $k$,~(\ref{eqBSDEiterating}) is a
standard BSDE with a driver $\lambda_t F$, which has Lipshitz
coefficients of $\bar c_t$ and $\bar c$, and hence is (linearly) firmly
Lipschitz by~(\ref{eqcbound}). Hence, the existence of a unique
solution for each $k$ is guaranteed by Theorem~\ref{thmBSDEExist}.

For $k=B-1$,~(\ref{eqBSDEiterating}) agrees with (\ref
{eqBSDEdistorted}), and hence with the original BSDE~(\ref{eqBSDE}),
for all $t\in[t_k, t_{k+1}]$. It follows that the solution
$Y^{B-1}_{t}$ is a solution to our original BSDE on the interval
$[t_{B-1}, t_B]$. Similarly, for $k=B-2$, this argument then implies
that $Y^{B-2}_t$ is a solution to our original BSDE on the interval
$[t_{B-2}, t_{B-1}]$, etc.

We now piece together these solutions to define $Y_t = Y^k_t$ where
$t\in\ ]t_k, t_{k+1}]$, and similarly for $Z$. By an inductive argument,
we can see that this will solve the desired BSDE. Furthermore, this
solution will be unique, as the solution is unique on each subsection
$]t_k, t_{k+1}]$.
\end{pf}
\begin{remark}
We note that, even when $\mu_T\leq1$, the conditions of Theorem~\ref
{thmBSDEExist2} are strictly weaker than those of Theorem \ref
{thmBSDEExist}. In this case, the jumps of $\mu$ satisfy $\Delta\mu
\leq1$, and it follows that a quadratic firm Lipschitz bound is weaker
than a linear firm Lipschitz bound.
\end{remark}
\begin{remark}
Clearly if $\Delta\mu=0$, then the requirement that $F$ is firmly
Lipschitz degenerates into the classical requirement that $F$ is
uniformly Lipschitz. It is to be expected that many of the
generalisations of the Lipschitz conditions which are known in the case
where our filtration is generated by a Brownian motion, that is, to
drivers with a stochastic Lipschitz bound, to drivers with quadratic
growth, to drivers with linear growth and a monotonicity condition,
etc., will also be possible in this situation. There is, however,
considerable difficulty involved in obtaining these results in the
simple continuous case, and it is to be expected that this difficulty
will be increased by the discontinuities present here.
\end{remark}
\begin{remark}
The situation where $F$ has stochastic Lipschitz bounds is of
particular interest here, as it would then be possible to consider
replacing~$\mu$ with a general predictable process of finite variation,
and consequently, with any square integrable special semimartingale.
Such a general situation is arguably as general as can be expected
within the context of stochastic integration.
\end{remark}

\section{A comparison theorem}
Given we have now established the existence of solutions to these
equations, we now wish to prove a comparison theorem for them. This is
based on the theorem in~\cite{Cohen2009}, for BSDEs of the type of~(\ref
{eqBSDEworthogonal}).\looseness=1

\begin{theorem}[(Comparison theorem)]\label{thmCompThm}
Suppose we have two BSDEs corresponding to standard coefficients and
terminal values $(F, Q)$ and $(\bar F, \bar Q)$. Let $(Y, Z)$ and
$(\bar Y, \bar Z)$ be the associated solutions. Suppose that for some
$s$, the following conditions hold:
\begin{longlist}[(iii)]
\item$Q\geq\bar Q$ $\mathbb{P}$-a.s.;
\item$\mu\times\mathbb{P}$-a.s. on $[s,T]\times\Omega$,
\[
F(\omega, u, \bar Y_{u-}, \bar Z_u) \geq\bar F(\omega, u, \bar Y_{u-},
\bar Z_u);
\]
\item for each $j$, there exists a measure $\tilde{\mathbb{P}}_j$
equivalent to $\mathbb{P}$ such that the $j$th component of $X$, as
defined for $r\geq s$ by
\begin{eqnarray*}
e_j^*X_r &:=& -\int_{]s,r]} e_j^*[F(\omega, u, \bar Y_{u-},
Z_u) - F(\omega, u, \bar Y_{u-}, \bar Z_u)]\,d\mu_u \\
&&{} + \sum_i\int
_{]s,r]}e_j^*[Z^i_u -\bar Z^i_u]\,dM_u^i
\end{eqnarray*}
is a $\tilde{\mathbb{P}}_j$ supermartingale on $[s,T]$;
\item if, for all $r\in[s,T]$,
\begin{eqnarray*}
&&e_i^*Y_r -E_{\tilde{\mathbb{P}}_i}\biggl[\int
_{]r,t]}e_i^*F(\omega, u, Y_{u-}, Z_u)\,d\mu_u\Big|\mathcal{F}_r\biggr]
\\
&&\qquad \geq e_i^*\bar Y_r-E_{\tilde{\mathbb{P}}_i}\biggl[\int
_{]r,t]}e_i^*F(\omega, u, \bar Y_{u-}, Z_u)\,d\mu_u\Big|\mathcal
{F}_r\biggr]
\end{eqnarray*}
for all $i$, then $Y_r\geq\bar Y_r$ for all $r\in[s,t]$ componentwise.

\end{longlist}
It is then true that $Y \geq\bar Y$ on $[s,T]$, except possibly on
some evanescent set.
\end{theorem}
\begin{pf}
We omit the $\omega$ and $t$ arguments of $F$ for clarity.

Then, for $r\in[s,T]$
%
\begin{eqnarray} \label{eqexpandBSDE}
&&Y_r-\bar Y_r - \int_{]r,T]} [F(Y_{u-}, Z_u) - \bar F(\bar Y_{u-}, \bar
Z_{u})]\,d\mu_u\nonumber\\
&&\quad{} + \sum_i\int_{]r,T]} [Z_u^i-\bar
Z_u^i]\,dM_u^i\\
&&\qquad=
Y_\tau-\bar Y_\tau\geq0.\nonumber
\end{eqnarray}

This can be rearranged to give
%
\begin{eqnarray}\label{eqsplitup}
&&Y_r - \bar Y_r - \int_{]r,T]} [F(Y_{u-}, Z_{u})-F(\bar Y_{u-},
Z_{u})]\,d\mu_u\nonumber\\
&&\qquad\geq\int_{]r,T]}[F(\bar Y_{u-}, \bar Z_{u})-\bar F(\bar Y_{u-}, \bar
Z_{u})]\,d\mu_u \nonumber\vadjust{\goodbreak}\\[-8pt]\\[-8pt]
&&\qquad\quad{} +\int_{]r,T]}[F(\bar Y_{u-}, Z_{u})-F(\bar Y_{u-}, \bar
Z_{u})]\,d\mu_u\nonumber\\
&&\qquad\quad{}-\sum_i\int_{]r,T]} [Z_u^i-\bar Z_u^i]\,dM_u^i.
\nonumber
\end{eqnarray}

We have that
\[
\int_{]r,T]}[F(\bar Y_{u-}, \bar Z_{u})-\bar F(\bar Y_{u-}, \bar
Z_{u})]\,d\mu_u\geq0
\]
by assumption (ii). As $e_j^*X_r$ is a $\tilde{\mathbb{P}}_j$
supermartingale, we know that the process given by
%
\begin{eqnarray}\label{eqXtildeDefn}
e_j^*\tilde X_{r} :\!&=&e_j^*X_{r}-E_{\tilde{\mathbb
{P}}_i}[e_j^*X_T|\mathcal{F}_r]\nonumber\\
&=& E_{\tilde{\mathbb{P}}_j}\biggl[\int_{]r,T]} e_j^*[F(\bar Y_{u-}, Z_u)
- F(\bar Y_{u-},\bar Z_u)]\,d\mu_u\\
&&\hspace*{56.1pt}{}  - \sum_i\int_{]r,T]}e_j^*[Z^i_u -\bar Z^i_u]\,dM_u^i\Big|\mathcal
{F}_r\biggr]
\nonumber
\end{eqnarray}
is also a $\tilde{\mathbb{P}}_j$-supermartingale, with $e_j^*\tilde
X_T=0$ $\tilde{\mathbb{P}}_j$-a.s. Hence $e_j^*\tilde X_r\geq0$.

For each $j$, taking a $\tilde{\mathbb{P}}_j|\mathcal{F}_r$ conditional
expectation throughout~(\ref{eqsplitup}) and premultiplying by $e_j^*$ gives
\[
e_j^*Y_r - e_j^*\bar Y_r - E_{\tilde{\mathbb{P}}_j}\biggl[\int
_{]r,T]} e_j^*[F(Y_{u-}, Z_u)-F(\bar Y_{u-}, Z_u)]\,d\mu_u\Big|\mathcal
{F}_r\biggr]\geq0.
\]
This must hold for all $r\in[s,T]$ and almost all $\omega$. By
assumption (iv), for almost all $\omega$, it follows that the
comparison $Y_r\geq\bar Y_r$ must hold for all $r\in[s,T]$.

As $Y-\bar Y$ is c\`adl\`ag, we have that $Y-Y$ is indistinguishable
from a~nonnegative process and, therefore, the inequality holds up to
evanescence.
\end{pf}
\begin{remark}
Assumption (iv) is clearly trivial whenever $F$ does not depend on $Y$.
\end{remark}
\begin{remark}
Assumption (iii) is very closely related to the \textit{Fundamental
theorem of asset pricing} (see~\cite{Delbaen2006}), as it relates an
inequality in current values to the existence of an equivalent
(super-)martingale measure.\vadjust{\goodbreak}
\end{remark}
\begin{corollary}
If assumption \textup{(iv)} holds for any $T$ whenever $s\geq T-\varepsilon$ for
some fixed $\varepsilon$, then the comparison also holds.
\end{corollary}
\begin{pf}
In this\vspace*{1pt} case, we can show that the comparison holds on
$[T-\varepsilon,\break
T]$. We can then replace $T$ with $T-\varepsilon$ throughout the theorem,
replacing~$Q$ and $\bar Q$ with $Y_{T-\varepsilon}$ and $\bar Y_{T-\varepsilon
}$ in assumption (i). It is clear that assumptions~(ii) and (iii) will
continue to hold, with the same choice of measures $\tilde{\mathbb
{P}}_i$. By the statement of the corollary, assumption (iv) will then
hold on the interval $[T-2\varepsilon, T-\varepsilon]$. By induction, it
follows that the comparison holds on $[T-n\varepsilon, T]$ for all $n\in
\mathbb{N}$. For $n$ sufficiently large, this implies the comparison
holds on $[s,T]$ as desired.
\end{pf}
\begin{definition}
A standard driver $F$ such that assumptions (iii)
and~(iv) of Theorem~\ref{thmCompThm} hold on $[0,T]$ for all $Y, \bar
Y \in S^2$ and $Z, \bar Z\in H^2_M$ will be called \textit{balanced}.
\end{definition}
\begin{theorem}\label{thmScalarsimple}
In the scalar ($K=1$) case, assumption \textup{(iv)} of Theorem~\ref
{thmCompThm} holds for any standard $F$.
\end{theorem}
\begin{pf}
As we are in the scalar case, we can omit the $e_i$ from the statement
of the assumption. Hence, we wish to show that, given for all $r\in[s,T]$
\begin{eqnarray*}
&&Y_r -E_{\tilde{\mathbb{P}}}\biggl[\int
_{]r,T]}F(\omega, u, Y_{u-}, Z_u)\,d\mu_u\Big|\mathcal{F}_r\biggr] \\
&&\qquad \geq\bar Y_r-E_{\tilde{\mathbb{P}}}\biggl[\int
_{]r,T]}F(\omega, u, \bar Y_{u-}, Z_u)\,d\mu_u\Big|\mathcal{F}_r\biggr],
\end{eqnarray*}
we must have $Y_r\geq\bar Y_r$. For simplicity, let $\delta Y:= Y-\bar Y$.

It is clear from the problem and the recursivity of BSDE solutions
that we can replace $T$ with any stopping time $\tau\leq T$ such that
$\delta Y_\tau\geq0$. By applying Lem\-ma~\ref{lemsplitupinterval}, we
can also assume that $s$ is such that $\int_{]s,T]} c_u \,d\mu_u <1$, and
simply piece together the result for general $s$.

Suppose on some nonnull set $A\in\mathcal{F}$, $\delta Y_u<0$ for some
$u\in[s,T]$. As $\delta Y$ is adapted and right continuous, this
implies that
there are stopping times~$\sigma, \tau$ such that $\delta Y_u<0$ for all
$u\in[\sigma,\tau[$, and $s\leq\sigma<\tau$ on $A$. Without loss of generality,
let $\tau$ be the largest such upper bound. Then, as $\delta Y_T\geq0$ and
$\tau\leq T$, it follows that $\delta Y_\tau\geq0$. Replacing $T$ with
$\tau$
in the above inequality, we know that
\begin{eqnarray*}
&&E_{\tilde{\mathbb{P}}}\bigl[I_{r\in[\sigma,\tau[} |\delta Y_r|\bigr]\\
&&\qquad= E_{\tilde{\mathbb{P}}}\bigl[-I_{r\in[\sigma,\tau[} \delta Y_r\bigr]\\
&&\qquad\leq E_{\tilde{\mathbb{P}}}\biggl[-I_{r\in[\sigma,\tau[} \int_{]r,\tau
]} F^1(\omega, u, Y^1_{u-}, Z^1_u) - F^1(\omega, u, Y^2_{u-}, Z^1_u)
\,d\mu_u\biggr]\\
&&\qquad\leq E_{\tilde{\mathbb{P}}}\biggl[I_{r\in[\sigma,\tau[}\int_{]r,\tau
]}c_u|\delta Y_{u-}|\,d\mu_u\biggr]\\
&&\qquad\leq \int_{]r,T]}E_{\tilde{\mathbb{P}}}\bigl[I_{u\in[\sigma,\tau
[}|\delta Y_{u-}|\bigr]c_u\,d\mu_u.
\end{eqnarray*}
Taking a left limit in $r$, we see
\[
E_{\tilde{\mathbb{P}}}\bigl[I_{r\in[\sigma,\tau[} |\delta Y_{r-}|\bigr] \leq
\int_{[r,T]}E_{\tilde{\mathbb{P}}}\bigl[I_{u\in[\sigma,\tau[}|\delta
Y_{u-}|\bigr]c_u\,d\mu_u.
\]
By assumption, this quantity is strictly positive. Integration on $]t,T]$
and Fubini's theorem gives, for $t>s$,
\begin{eqnarray*}
\int_{]t,T]} E_{\tilde{\mathbb{P}}}\bigl[I_{r\in[\sigma,\tau[} |\delta
Y_{r-}|\bigr]c_r\,d\mu_r
&\leq& \int_{]t,T]}\biggl(\int_{[r,T]}E_{\tilde{\mathbb{P}}}\bigl[I_{u\in
[\sigma,\tau[}|\delta Y_{u-}|\bigr]c_u\,d\mu_u\biggr) c_r\,d\mu_r\\
&=& \int_{]t,T]}\biggl(\int_{]t,u]} c_r\,d\mu_r\biggr) E_{\tilde{\mathbb
{P}}}\bigl[I_{u\in[\sigma,\tau[}|\delta Y_{u-}|\bigr]c_u\,d\mu_u\\
&<& \int_{]t,T]} E_{\tilde{\mathbb{P}}}\bigl[I_{u\in[\sigma,\tau[}|\delta
Y_{u-}|\bigr]c_u\,d\mu_u,
\end{eqnarray*}
where the last line is due to our\vspace*{1pt} assumption that $\int_{]s,T]}c_t \,d\mu
_t <1$.
This contradicts our assumption that this quantity is strictly positive.
Therefore, $A$ is a null set, that is, $\delta Y_u \geq0$ for all $u\in[s,t]$.
\end{pf}
\begin{definition}
The comparison between $Y$ and $\bar Y$ will be
called \textit{strict} on $[s,T]$ if the conditions of Theorem \ref
{thmCompThm} hold, and, for any $A\in\mathcal{F}_s$ such that
$Y_s=\bar Y_s$ $\mathbb{P}$-a.s. on $A$, we have $Y_u = \bar Y_u$ on
$[s,T]\times A$, up to evanescence.
\end{definition}
\begin{lemma}
If the comparison is strict on $[s,T]$, then for any $A\in\mathcal
{F}_s$ such that $Y_s=\bar Y_s$ $\mathbb{P}$-a.s. on $A$, it follows that:
\begin{itemize}
\item$Q=\bar Q$ $\mathbb{P}$-a.s. on $A$;
\item$F(\omega, u, \bar Y_{u-}, \bar Z_u)= \bar F(\omega, u, \bar
Y_{u-}, \bar Z_u)$ $\mu\times\mathbb{P}$-a.s. on $[s,T]\times A$;
\item$Z\equiv\bar Z$ in $H^2_M$ on $[s,T]\times A$.
\end{itemize}
\end{lemma}
\begin{pf}
We omit the $\omega$ and $t$ arguments of $F$ and $\bar F$ for clarity.
Let $\tilde X$ be as in~(\ref{eqXtildeDefn}), and let $S$ be the
process defined by
%
\begin{eqnarray}\label{eqSdefn}\quad
e_j^*S_r&:=&e_j^*E_{\tilde{\mathbb{P}}_i}[
Q-\bar Q|\mathcal{F}_r] \nonumber\\[-8pt]\\[-8pt]
&&{} + e_j^*E_{\tilde{\mathbb
{P}}_i}\biggl[\int_{]r,T]} [F(\bar Y_{u-}, \bar Z_{u}) - \bar
F(\bar Y_{u-}, \bar Z_{u})]\,d\mu_u\Big|\mathcal{F}_r\biggr] +
e_j^*\tilde X_r.
\nonumber
\end{eqnarray}
Then $e_j^*S$ is a $\tilde{\mathbb{P}}_j$-supermartingale, as the first
term is a $\tilde{\mathbb{P}}_j$-martingale, the second is
nonincreasing in $r$ by assumption (ii) of Theorem~\ref{thmCompThm}
and the third is a $\tilde{\mathbb{P}}_j$-supermartingale by assumption
(iii) of Theorem~\ref{thmCompThm}. Furthermore, each of these terms is
nonnegative.

Taking a $\tilde{\mathbb{P}}_j|\mathcal{F}_r$ conditional expectation
through~(\ref{eqBSDE}), we have that, for all $r\in[s,T]$,
%
\begin{equation}\label{eqSBSDE}
e_j^*(Y_r - \bar Y_r) = e_j^*S_r + E_{\tilde{\mathbb{P}}}\biggl[
\int_{]r,t]}e_j^*[F(Y_{u-}, Z_u) - F(\bar Y_{u-}, Z_u)]\,d\mu_u
\Big|\mathcal{F}_r\biggr].\hspace*{-35pt}
\end{equation}

If $Y_r=\bar Y_r$ on $[s,T]\times A$ up to evanescence, then it is
clear from~(\ref{eqSBSDE}) that $S_ r=0$ $\mathbb{P}$-a.s. on
$[s,T]\times A$. Hence, by nonnegativity, each of the terms on the
right-hand side of~(\ref{eqSdefn}) must be zero. The first two points
of the lemma immediately follow.

Consider the BSDE~(\ref{eqBSDE}) satisfied by $\bar Y$. As $F(\bar
Y_{u-}, Z_u)=\bar F(Y_{u-}, Z_u)$ $\mu\times\mathbb{P}$-a.s. on
$[s,T]\times A$ and $Q=\bar Q$ $\mathbb{P}$-a.s. on $A$, we know that
\[
\bar Y_r - \int_{]r,T]} \bar F(\bar Y_{u-}, \bar Z_u) \,d\mu_u+ \sum_i\int
_{]r,T]} \bar Z^i_{u}\,dM_u^i= \bar Q
\]
is $\mathbb{P}$-a.s. equal to
\[
\bar Y_r - \int_{]r,T]} F(\bar Y_{u-}, \bar Z_u) \,d\mu_u+ \sum_i\int
_{]r,T]} \bar Z^i_{u}\,dM_u^i= Q.
\]
Hence, in $A$, $(\bar Y, \bar Z)$ is a solution at time $r$ to the BSDE
defining $(Y, Z)$.

As the solution to this BSDE is unique, it follows that, on
$[s,T]\times A$, $\bar Z\equiv Z$ in~$H^2_{M_t}$.
\end{pf}
\begin{theorem}[(Strict comparison)] \label{thmStrictComp1}
Consider the scalar ($K=1$) case, where $F$ is balanced. Then the
comparison is strict on $[s,T]$ for all $s$.
\end{theorem}
\begin{pf}
Again, as $K=1$ we can omit $e_j$ from all equations, and we omit the
$\omega$ and $t$ arguments of $F$ and $\bar F$ for clarity. Let $S_r$
be as defined in~(\ref{eqSdefn}), and note that $S$ is a nonnegative
$\tilde{\mathbb{P}}$-supermartingale.

Taking a $\tilde{\mathbb{P}}|\mathcal{F}_s$ conditional expectation of
(\ref{eqSBSDE}) gives
%
\begin{eqnarray}\label{eqYSinequality}
E_{\tilde{\mathbb{P}}}[Y_r - \bar Y_r|\mathcal
{F}_s]
&=& E_{\tilde{\mathbb{P}}}\biggl[S_r + \int_{]s,t]}[F(Y_{u-}, Z_u)
- F(\bar Y_{u-}, Z_u)]\,d\mu_u\Big|\mathcal{F}_s\biggr]\nonumber\\
&&{} - E_{\tilde{\mathbb{P}}}\biggl[\int_{]s,r]}[F(Y_{u-}, Z_u)
- F(\bar Y_{u-}, Z_u)]\,d\mu_u\Big|\mathcal{F}_s\biggr]\nonumber\\
&\leq& S_s+E_{\tilde{\mathbb{P}}}\biggl[\int_{]s,t]}[F(Y_{u-}, Z_u)
- F(\bar Y_{u-},
Z_u)]\,d\mu_u\Big|\mathcal{F}_s\biggr]\nonumber\\[-8pt]\\[-8pt]
&&{} +\int_{]s,r]}E_{\tilde{\mathbb{P}}}[|F(Y_{u-}, Z_u)
- F(\bar Y_{u-}, Z_u)||\mathcal{F}_s]\,du\nonumber\\
&\leq& S_s+E_{\tilde{\mathbb{P}}}\biggl[\int_{]s,t]}[F(Y_{u-}, Z_u)
- F(\bar Y_{u-}, Z_u)]\,d\mu_u\Big|\mathcal{F}_s\biggr]\nonumber\\
&&{} +c \int_{]s,r]}E_{\tilde{\mathbb{P}}}[|Y_{u-}- \bar
Y_{u-}||\mathcal{F}_s]\,d\mu_u.
\nonumber
\end{eqnarray}

We know from~(\ref{eqSBSDE}) and the assumption $Y_s-\bar Y_s=0$ on
$A$ that
\[
I_AS_s+I_AE_{\tilde{\mathbb{P}}}\biggl[\int_{]s,t]}[F(Y_{u-}, Z_u)
- F(\bar Y_{u-}, Z_u)]\,d\mu_u\Big|\mathcal{F}_s\biggr]=I_A(Y_s-\bar
Y_s) = 0,
\]
and so, as $Y-\bar Y$ is nonnegative by Theorem~\ref{thmCompThm},
premultiplication of~(\ref{eqYSinequality}) by $I_A$ and then taking
an expectation gives
\[
E_{\tilde{\mathbb{P}}}[I_A(Y_r - \bar Y_r)]\leq c \int_{]s,r]}E_{\tilde
{\mathbb{P}}}[I_A(Y_{u-}- \bar Y_{u-})]\,d\mu_u.
\]
As all quantities are nonnegative, taking a limit from below yields
\[
E_{\tilde{\mathbb{P}}}[I_A(Y_{r-} - \bar Y_{r-})]\leq c \int
_{]s,r]}E_{\tilde{\mathbb{P}}}[I_A(Y_{u-}- \bar Y_{u-})]\,d\mu_u,
\]
and an application of (the forward version of) Gr\"onwall's lemma implies
\[
E_{\tilde{\mathbb{P}}}[I_A(Y_r - \bar Y_r)]\leq0.
\]
By nonnegativity, it follows that $Y_r=\bar Y_r$, $\tilde{\mathbb
{P}}$-a.s. on $A$. Again, as $Y-\bar Y$ is c\`adl\`ag, this shows that
$Y=\bar Y$ on $[s,t]\times A$, up to evanescence.
\end{pf}
\begin{corollary} \label{corstrictvectcomp}
$\!\!\!$If the $i$th component of $F(\omega, t, y, z)$ depends only~on~the
$i$th component of $y$ (as well as on $\omega, t$ and $z$), then the
comparison is strict.
\end{corollary}
\begin{pf}
As the $i$th component of $F$ depends only on the $i$th component
of~$y$, we can repeat the construction of Theorem~\ref{thmStrictComp1} in
each component. The result follows.
\end{pf}
\begin{remark}
In the scalar case, with a simple Brownian filtration ($M^1=W$, $M^i=0$
for $i\geq2$) and $d\mu=dt$, we can use Girsanov's transformation to
construct the measure required for assumption (iii) of Theorem \ref
{thmCompThm}. We write
\[
\Lambda_t := 1+ \int_{]0,t]}\Lambda_{u-}\frac{F(\omega, u, \bar Y_{u-},
Z_u)- F(\omega, u, \bar Y_{u-}, \bar Z_u)}{Z_u-\bar Z_u} \,dW_u,
\]
then $d\tilde{\mathbb{P}}/d\mathbb{P} = \Lambda_T$. It is then easy to
verify that $X$ is a martingale. In this case, using Theorem \ref
{thmScalarsimple} we can see that any Lipschitz continuous $F$ is balanced.
\end{remark}

\section{Nonlinear expectations}
We are now in a position to explicitly construct nonlinear expectations
in a general probability space. We shall not here consider the more
general theory of nonlinear evaluations. An approach without these\vadjust{\goodbreak}
restrictions can be seen in~\cite{Cohen2009}. These operators,
discussed in~\cite{Peng2005}, are closely related to the theory of
dynamic risk measures, as in~\cite{Barrieu2004,Rosazza2006,Artzner2007} and others, as each concave nonlinear expectation
$\mathcal{E}(\cdot|\mathcal{F}_t)$ corresponds to a dynamic convex risk
measure through the relationship
\[
\rho_t(Q) = -\mathcal{E}(Q|\mathcal{F}_t).
\]
A further discussion of this relationship can be found in~\cite{Rosazza2006}.
\begin{definition}
A family of operators
\[
\mathcal{E}(\cdot|\mathcal{F}_t)\dvtx L^2(\mathcal{F}_T)\rightarrow
L^2(\mathcal{F}_t),\qquad 0\leq t \leq T,
\]
is called an $\mathcal{F}_t$-consistent \textit{nonlinear expectation}
if $\mathcal{E}(\cdot|\mathcal{F}_t)$ satisfies the following properties:

\begin{longlist}[(1)]
\item[(1)] If $Q\geq\bar Q$ $\mathbb{P}$-a.s. componentwise
\[
\mathcal{E}(Q|\mathcal{F}_t) \geq\mathcal{E}(\bar
Q|\mathcal{F}_t),\qquad
\mathbb{P}\mbox{-a.s. componentwise}
\]
with equality iff $Q = \bar Q$ $\mathbb{P}$-a.s.
\item[(2)] For $Q\in L^2(\mathcal{F}_t)$, $\mathcal{E}(Q|\mathcal{F}_t) = Q$
$\mathbb{P}$-a.s.
\item[(3)] For any $s\leq t$,
\[
\mathcal{E}(\mathcal{E}(Q|\mathcal{F}_t)|\mathcal{F}_s) = \mathcal
{E}(Q|\mathcal{F}_s), \qquad\mathbb{P}\mbox{-a.s.}
\]
\item[(4)] For any $A\in\mathcal{F}_t$,
\[
I_A\mathcal{E}(Q|\mathcal{F}_t) =
\mathcal{E}(I_AQ|\mathcal{F}_t),\qquad
\mathbb{P}\mbox{-a.s.}
\]
\end{longlist}
\end{definition}
\begin{theorem}
Let $F$ be a balanced driver which does not depend on~$Y$ (i.e.,
$c_t\equiv0$) and satisfies $F(\omega, t, y, 0) =0$ $\mu\times\mathbb
{P}$-a.s. Then the operator defined by
\[
\mathcal{E}(Q|\mathcal{F}_t) = Y_t,
\]
where $Y$ is the solution to a BSDE~(\ref{eqBSDE}) with driver $F$, is
a nonlinear expectation.
\end{theorem}
\begin{pf}
(1) As $F$ is balanced, this result follows directly from the
comparison theorem (Theorem~\ref{thmCompThm}). As $F$ does not depend
on $Y$, the strict comparison will also hold, by Corollary \ref
{corstrictvectcomp}.

(2) Consider the BSDE~(\ref{eqBSDE}) on $[t,T]$
\[
Y_s -\int_{]s,T]} F(\omega, u, Y_{u-}, Z_u)\,d\mu_u +\sum_i\int_{]s,T]}
Z_u^i \,dM^i_u = Q.
\]
This has a solution $Y_s=Q$, $Z_s=0$. As $Q\in L^2(\mathcal{F}_t)$,
this solution is adapted and, by Theorem~\ref{thmBSDEExist}, unique.
Therefore $\mathcal{E}(Q|\mathcal{F}_t)=Y_t=Q$ as desired.

(3) By definition the BSDE with terminal condition $Q$ at time $T$
has solution $Y_t$ at time $t$. Simple manipulation of the BSDE (\ref
{eqBSDE}) at time $s$ shows that $Y_s$ is also the time $s$ solution
to the BSDE with terminal condition $Y_t$ at time $t$. Hence, by
property 2, $Y_s$ solves both the BSDE with terminal condition
$Y_t=\mathcal{E}(Q|\mathcal{F}_t)$ and the BSDE with terminal condition
$Q$.\vadjust{\goodbreak}

(4) Consider the BSDE with driver $F$ and terminal condition $Q$.
Multiplying by~$I_A$, as $I_AF(\omega, t, y, z) = F(\omega, t, I_Ay,
I_A z)$, we see that $(I_A Y, I_A Z)$ is the solution to the BSDE with
driver $F$ and terminal condition $I_AQ$, as desired.
\end{pf}
\begin{remark}
It is known in discrete time~\cite{Cohen2008c}, and under some
conditions in continuous time~\cite{Coquet2002}, that BSDEs describe
all nonlinear expectations, subject to some boundedness conditions. It
is likely that a similar result will hold in this setting. However,
obtaining such a result is beyond the scope of this paper.
\end{remark}

\section{Conclusions}
We have constructed BSDEs in a general filtered probability space,
using only basic properties of the filtration. We have presented
conditions for the existence of unique solutions to these equations,
and seen how these are related to the conditions in both the classical
setting, and the discrete time setting. We have given a comparison
theorem for these~solutions, which allows the construction of nonlinear
\mbox{expectations in these~spaces.}

These results are significantly more general than those previously
available, as they make very few assumptions on the underlying
probability space. A consequence of this is that a possibly
infinite-dimensional martingale representation theorem is required. In
full generality, they also make no assumptions regarding the
relationship of the integrator of the driver and the quadratic
variations of the martingale terms. At the same time, this general
setting provides an approach unifying the theory of BSDEs in discrete
and continuous time.\vspace*{-3pt}

\section*{Acknowledgment}
Robert Elliott wishes to thank the Australian Research Council for support.\vspace*{-3pt}



%
\printaddresses

\end{document}